\documentclass[11pt,a4paper,reqno]{amsart}
\usepackage{amsmath,amssymb, amsbsy}
\usepackage{enumitem}
\usepackage{hyperref}
\usepackage{color,psfrag}
\usepackage[dvips]{graphicx}
\usepackage[arrowdel]{physics}
\usepackage{color}

\newcommand{\noi} {\noindent}
\theoremstyle{plain}
\newtheorem{thm}{Theorem}[section]
\theoremstyle{plain}
\newtheorem{lem}[thm]{Lemma}
\newtheorem{prop}[thm]{Proposition}

\theoremstyle{definition}

\newtheorem{rem}{Remark}[section]
\newtheorem{thmm}{Theorem}[section]

\newcommand{\De} {\Delta}
\newcommand{\la} {\lambda}
\newcommand{\rn}{\mathbb{R}^{N}}
\newcommand{\bn}{\mathbb{B}^{N}}

\newcommand{\authorfootnotes}{\renewcommand\thefootnote{\@fnsymbol\c@footnote}}%

\def\e{{\text{e}}}

\numberwithin{equation}{section} \allowdisplaybreaks

\usepackage[text={6in,8.6in},centering]{geometry}
\begin{document}\title[Elliptic equations on hyperbolic space]
        {
        Existence of high energy positive solutions for a class of elliptic equations in the hyperbolic space}

  \author[Debdip Ganguly]{Debdip Ganguly}
\address{ Department of Mathematics, Indian Institute of Technology Delhi, Hauz Khas New Delhi 110016,  India}
\email{debdip@maths.iitd.ac.in}

\author[Diksha Gupta]{Diksha Gupta}
\address{ Department of Mathematics, Indian Institute of Technology Delhi, Hauz Khas New Delhi 110016,  India}
\email{dikshagupta1232@gmail.com}

\author[ K.~Sreenadh]{K.~Sreenadh}
\address{ Department of Mathematics, Indian Institute of Technology Delhi, Hauz Khas New Delhi 110016,  India}
\email{sreenadh@maths.iitd.ac.in}

    \date{\today}

   \begin{abstract}
We study the existence  of positive solutions for the following class of scalar field problem on the hyperbolic space
\begin{equation*}
		-\Delta_{\mathbb{H}^N} u - \lambda u = a(x) |u|^{p-1} \, u\;\;\text{in}\;\mathbb{B}^{N}, \quad
		u \in H^{1}{(\mathbb{B}^{N})},
\end{equation*}
where $\mathbb{B}^N$ denotes the hyperbolic space,  $1<p<2^*-1:=\frac{N+2}{N-2}$, if $N \geqslant 3; 1<p<+\infty$, if $N = 2,\;\lambda < \frac{(N-1)^2}{4}$, and $0< a\in L^\infty(\mathbb{B}^N).$ We prove the existence of a positive solution by introducing the min-max procedure  in the spirit of Bahri-Li in the hyperbolic space and using a series of new estimates involving 
{\it  interacting hyperbolic bubbles. \rm}
\end{abstract}

\maketitle

\section{Introduction}
\noi In this paper, we investigate the existence of positive solutions for the following semilinear elliptic problem on the hyperbolic space $\mathbb{B}^{N}$
\begin{equation*}
(P_\la) \;\; \left\{ -\Delta_{\mathbb{B}^{N}} u \, - \, \lambda u \, = \, a(x) \, |u|^{p-1} u, \quad u \in H^{1}\left(\mathbb{B}^{N}\right), \right.
\end{equation*}
where  $1<p<2^*-1:=\frac{N+2}{N-2}$, if $N \geqslant 3; 1<p<+\infty$, if $N = 2,\;\lambda < \frac{(N-1)^2}{4},$ $H^{1}\left(\mathbb{B}^{N}\right)$ denotes the Sobolev space 
on the disc model of the hyperbolic space $\mathbb{B}^{N},$ $\Delta_{\mathbb{B}^{N}}$ denotes the Laplace Beltrami operator on $\mathbb{B}^{N},$ $\frac{(N-1)^2}{4}$ 
being the bottom of the $L^2-$ spectrum of $-\Delta_{\bn},$ and $a(x) \in L^{\infty}(\bn).$ Moreover, we investigate the existence of solutions under the following hypotheses
\medskip

\begin{enumerate}
\item[(a1)] $ a(x) > 0, \forall \ x \in \mathbb{\bn}$ \ and \   $\displaystyle \lim_{d(x, 0) \rightarrow \infty} a(x)  =  1,$
 \item [(a2)] there exist some positive constants $C$ and $\delta > 0$ such that
	$$a(x) \geqslant 1 -\operatorname{C \, exp}(-\delta \, d(x,0))  \quad \forall \ d(x, 0) \rightarrow \infty,$$
\end{enumerate}
\noi where $d$ is the hyperbolic distance. 

 \noi The problem, $(P_\la)$ when posed in all of $\mathbb{R}^N,$ with $\lambda = -1,$ has been the subject of intense research in the past few decades, starting from the seminal 
papers by Berestycki-Lions \cite{BL1, BL2}, Bahri-Berestycki \cite{BB1}, Bahri-Li \cite{Bahri-Li}, Bahri-Lions \cite{Bahri}. Apart from purely mathematical interest, such semilinear elliptic equations  also arise in several 
physical phenomena, e.g., nonlinear Schr\"odinger and Klein-Gordan equations (see \cite{FW, YGO1, YGO2}). The main difficulty in studying such equations
 in all of $\mathbb{R}^N$ comes from the lack of compactness due to the  unboundedness of the domain $\mathbb{R}^N,$ i.e., through the vanishing of mass in the sense of concentration compactness of Lions \cite{Lions}. Non-compact subcritical problems have been studied thoroughly and brought to a high level of sophistication by many authors.  We refer to the papers \cite{AD, Adachi, AD2, CP, CW,  DI, DN,  AM, MMP, MP1} and references  quoted therein without any claim of completeness. The asymptotic of $a(x)$ in $(a1)$ suggests that the corresponding ``limiting problem," i.e., 
 \begin{equation}\label{euc}
 -\Delta_{\mathbb{R}^N}u \, + \, u \, = \, u^p \quad \mbox{in} \ \mathbb{R}^N, \quad u \in H^{1}(\mathbb{R}^N), \quad u > 0  \quad \mbox{in} \ \mathbb{R}^N,
 \end{equation}
will play an essential role. In fact, the proof of the existence of positive solutions  introduced by Bahri-Li in \cite{Bahri-Li} essentially depends on the uniqueness (up to translation) and decay 
estimate of the unique solution, and it is worth mentioning that Bahri-Li's  solution to the equation $(P_\la)$ does not correspond to the mountain pass critical point. More precisely, consider 
the energy functional $E_a : H^1(\mathbb{R}^N) \rightarrow \mathbb{R}$ defined by 
$$
E_a(u) := \frac{\int_{\mathbb{R}^N} \left( |\nabla u|^2) \, + \,  u^2 \right) \, {\rm d}x}{\left( \int_{\mathbb{R}^N} a(x) \, |u|^{p+1} \, {\rm d}x \right)^{\frac{2}{p+1}}}.
$$
and  it was shown in \cite{Bahri-Li} that the Palais-Smale condition holds for $E_a$ in the range $(S_1, 2^{\frac{p-1}{p+1}} S_1),$ where 
$$
S_1 :=  \inf_{u \in H^1(\mathbb{R}^N)\backslash\{0\}}\,  E_a(u). 
$$
 We note that if $a(x)\ge 1$ in $\mathbb R^N,$ then the mountain pass level for $E_a$ is strictly less than $S_1$ and hence achieved (see \cite{AD}). But assumptions (a1) and 
	(a2), in general, do not ensure  $a(x) \ge 1.$ So Bahri-Li constructed solutions in higher energy level by using  energy estimates involving interacting (translated) solutions to the limiting problem.
\medskip 

\noi As concerns the analogous equation in the hyperbolic space, Sandeep-Mancini, in their seminal paper \cite{MS}, have shown 
  the existence and uniqueness of finite energy positive solutions of the following homogeneous elliptic equation: 
\begin{equation}\label{hsm}
-\Delta_{\mathbb{B}^{N}} u \, - \, \lambda u \, = \, u^{p}, \quad u \in H^{1}\left(\mathbb{B}^{N}\right),
\end{equation}
where $\lambda \leq \frac{(N-1)^2}{4},$ $1 < p \leq \frac{N+2}{N-2}$ if $N \geq 3; $ $1 < p < \infty$ if $N=2.$ Among many other results, they established in the subcritical 
case, and for $p > 1,$ if $N =2$ and $1 < p < 2^{\star} -1$ if $N \geq 3,$  the problem \eqref{hsm} has a positive solution if and only if $\lambda < \frac{(N-1)^2}{4}.$ These 
positive solutions are also shown to be unique up to hyperbolic isometries, except possibly for $N =2$ and $\lambda > \frac{2(p+1)}{(p+3)^2}.$ The question of finite energy solutions was fully resolved in \cite{MS} and subsequently, the existence of sign-changing solutions, compactness, and non-degeneracy were studied in \cite{BS, DG1, DG2}
and on the other hand, authors in \cite{BFG, BGGV}  showed the existence of infinite energy solutions of \eqref{hsm} and determined the exact asymptotic behavior of wide classes of finite and infinite energy solutions.

\medskip 
\noi One should notice that for the existence of finite energy solutions, one looks for the energy functional associated with \eqref{hsm} which is  defined as 

$$
E_{\lambda} (u) = \frac{1}{2} \int_{\bn} \left( |\nabla_{\bn} u|^2 - \lambda u^2 \right) \, \mathrm{~d} V_{\mathbb{B}^{N}} \, - \, \frac{1}{p+1} \int_{\bn} |u|^{p+1} \, \mathrm{~d} V_{\mathbb{B}^{N}}.
$$
In the subcritical case, the  variational problem lacks compactness because of the {\it hyperbolic translation\rm} (see section~\ref{SecPS} for more details), and so it cannot 
be solved by the standard minimization method.  Moreover, in \cite{BS}, a detailed analysis of the Palais-Smale decomposition is performed. One can easily see that if $U$ is 
a solution of \eqref{hsm}, then
\[u := U\circ\tau, \quad \mbox{for} \ \tau \in I(\bn),
\]
\noi where $I(\bn)$ is the group of isometries on the hyperbolic space,   is also a solution, and hence if we define a sequence by varying $\tau_n\in I(\bn),$ then $u_n$ is also a  Palais-Smale (PS) sequence for $E_{\lambda}.$ In fact, it was shown in 
\cite[Theorem~3.3]{BS} that in the subcritical case, i.e., when $1< p < \frac{N+2}{N-2},$ non-compact PS sequences
 are made of finitely many sequences of the form 
$u_n.$ We call this $u$ a {\it hyperbolic bubble \rm}(See Section~\ref{SecPS} below for more details). 
In our analysis, these {\it hyperbolic bubbles \rm}will play a major role. 
\medskip 

\noi Inspired by the above mentioned papers, our main goal in this article is to study whether positive solutions can still exist for a perturbed problem as in $(P_\la).$  
Moreover, we seek solutions in the {\it higher energy \rm}range following  Bahri-Li in 
\cite{Bahri-Li}. In particular, we prove the following theorem. 

\medskip

\begin{thmm}\label{maintheorem1} Assume $a(x)$ satisfies  (a1)- (a2). Then $(P_\la)$ admits a  positive solution for all $\la$ in the range
	\[\la \in \begin{cases}  \left( -\infty, \frac{2(p+1)}{(p+3)^2}\right], & N=2, \\ \left(-\infty, \frac{(N-1)^2}{4} \right), & N\ge 3.\end{cases}\]
\end{thmm}

%
%
%
%
%

%


\medskip

\subsection{Main Novelty and the strategy of the Proof}

As discussed earlier, we follow the approach of Bahri-Li. We first show that the energy functional associated with $(P_\la)$ satisfies the Palais-Smale condition in the range 
$(S_{1, \lambda}, S_{2, \lambda})$ where $S_{i, \lambda},\;i = 1, 2$ are defined in \eqref{3c}, and a detailed PS decomposition is provided in Proposition~\ref{prop1}. The proof of  proposition~\ref{prop1} is a starightforward adaptation of \cite[Theorem~3.3]{BS}  to the problem  $(P_\la)$. In the next step, we prove the key energy estimates. 
The main novelty lies in this step, where we compute the energy estimates involving the convex combination of translated {\it hyperbolic bubbles \rm} with two different centers. To this end,  
we need precise estimates on the interacting {\it hyperbolic bubbles, \rm} and since the hyperbolic volume grows exponentially, i.e., 

$$
 \mathrm{~d} V_{\mathbb{B}^{N}} \asymp e^{(N-1) r}, \quad r \rightarrow \infty,
$$
where $r$ denotes the geodesic distance $d(x, 0)$, and $\mathrm{~d} V_{\mathbb{B}^{N}}$ denotes the hyperbolic volume form, and 
to compensate for the exponential volume growth of the hyperbolic space, one requires a new way to tackle the integrals involving interacting bubbles. Unlike in the Euclidean space, 
where solutions to \eqref{euc} decays exponentially and the euclidean volume grows polynomially. We compute the effect of interacting bubbles in the 
integral in a novel way by breaking the integral in different sub-regions, and using the precise asymptotic estimate on the solution of \eqref{hsm} obtained in \cite{MS}.
We establish a series of new estimates for interacting terms, which are crucial in proving the corresponding energy of the
convex combination is strictly less than $S_{2, \lambda}.$ Finally, the existence of a positive solution is proved using the {\it min-max \rm} procedure of Bahri-Li by suitably defining the 
center of mass for functions in $H^1(\bn)$ on the hyperbolic space. We refer section~\ref{pf} below for more details.

\medskip

\noi The paper is organized as follows: In Section~\ref{sec2}, we introduce some of the notations and some geometric definitions and preliminaries concerning hyperbolic space.
In Section~\ref{SecPS}, we state and prove the Palais-Smale decomposition theorem associated with $(P_\la)$. Section~\ref{keylemma} contains the proof of key estimates 
on the energy, while Section~\ref{pf} contains the proof of the main Theorem~\ref{maintheorem1}.


 \section{Notations and Functional Analytic Preliminaries}\label{sec2}

\noi In this section, we will introduce some of the notations and definitions used in this
 paper and also recall some of the embeddings
  related to the Sobolev space on the hyperbolic space.  
  
  \medskip

\noi We will denote by $\bn$ the disc model of the hyperbolic space, i.e., the unit disc
 equipped with 
 the Riemannian metric $g_{\bn} := \sum\limits_{i=1}^N \left(\frac{2}{1-|x|^2}\right)^2 \, {\rm d}x_i^2$. To simplify our notations, we will denote $g_{\bn}$
by $g$.
 The corresponding volume element is given by $\mathrm{~d} V_{\mathbb{B}^{N}} = \big(\frac{2}{1-|x|^2}\big)^N {\rm d}x, $ where ${\rm d}x$ denotes the Lebesgue 
measure on $\rn$.  

\medskip 

\noi {\bf Hyperbolic distance on $\bn$.} The hyperbolic distance between two points $x$ and $y$ in $\bn$ will be denoted by $d(x, y).$ For the hyperbolic distance between
$x$ and the origin we write 

$$
\rho := \, d(x, 0) = \int_{0}^{r} \frac{2}{1 - s^2} \, {\rm d}s \, = \, \log \frac{1 + r}{1 - r},
$$
where $r = |x|$, which in turn implies that  $r = \tanh \frac{\rho}{2}.$ Moreover, the hyperbolic distance between $x, y \in \bn$ is given by 

$$
d(x, y) = \cosh^{-1} \left( 1 + \dfrac{2|x - y|^2}{(1 - |x|^2)(1 - |y|^2)} \right).
$$
It easily follows that a subset $S$ of $\bn$ is a hyperbolic sphere in $\bn$ iff $S$ is a Euclidean sphere in $\mathbb{R}^N$ and contained in $\bn$, probably 
with a different center and different radius, which can be computed. Geodesic balls in $\bn$ of radius $a$ centered at the origin will be denoted by

$$
B(0, a) : = \{ x \in \bn : d(x, 0) < a \}.
$$

\noi We also need some information on the isometries of $\bn$. Below we recall the
definition of a particular type of isometry, namely the hyperbolic translation. For
more details on the isometry group of $\bn$, we refer \cite{RAT}.

\medskip 

\noi {\bf Hyperbolic Translation.} For $b \in \bn,$ define 

\begin{equation}\label{hypt}
\tau_b(x) = \dfrac{(1 - |b|^2)x + (|x|^2 + 2x.b + 1)b}{|b|^2 |x|^2 + 2x.b + 1},
\end{equation}
then $\tau_b$ is an isometry of $\bn$ with $\tau_b(0) = b.$ The map 
$\tau_b$ is called the hyperbolic translation of $\bn$ by $b.$ It can also be seen that 
$\tau_{-b} = \tau_b^{-1}.$

\medskip

 \noi The hyperbolic gradient $\nabla_{\bn}$ and the hyperbolic Laplacian $\De_{\bn}$ are
 given by
 \begin{align*}
  \nabla_{\bn}=\left(\frac{1-|x|^2}{2}\right)^2\nabla,\ \ \ 
 \De_{\bn}=\left(\frac{1-|x|^2}{2}\right)^2\De + (N-2)\frac{1 - |x|^2}{2} \, \langle x, \nabla \,\rangle.
 \end{align*}

\medskip

\noi{\bf A sharp Poincar\'{e}-Sobolev inequality.} (see \cite{MS})

\medskip

\noi{\bf Sobolev Space :} We will denote by ${H^{1}}(\bn)$ the Sobolev space on the disc
model of the hyperbolic space $\bn$, equipped with norm $\|u\|=\left(\int_{\mathbb{B}^N} |\nabla_{\mathbb{B}^{N}} u|^{2}\right)^{\frac{1}{2}},$
where  $|\nabla_{\bn} u| $ is given by
$|\nabla_{\bn} u| := \langle \nabla_{\bn} u, \nabla_{\bn} u \rangle^{\frac{1}{2}}_{\bn} .$ \\

For $N \geq 3$ and every $p \in \left(1, \frac{N+2}{N-2} \right]$ there exists an optimal constant 
$S_{N,p} > 0$ such that
\begin{equation*}
 S_{N,p} \left( \int_{\mathbb{B}^{N}} |u|^{p + 1} \mathrm{~d} V_{\mathbb{B}^{N}} \right)^{\frac{2}{p + 1}} 
\leq \int_{\mathbb{B}^N} \left[|\nabla_{\mathbb{B}^{N}} u|^{2}
 - \frac{(N-1)^2}{4} u^{2}\right] \, \mathrm{~d} V_{\mathbb{B}^{N}},
\end{equation*}
for every $u \in C^{\infty}_{0}(\mathbb{B}^{N}).$ If $ N = 2$, then any $p > 1$ is allowed.

\noi A basic information is that the bottom of the spectrum of $- \Delta_{\bn}$ on $\bn$ is 
\begin{equation}\label{firsteigen}
  \frac{(N-1)^2}{4} = \inf_{u \in H^{1}(\bn)\setminus \{ 0 \}} 
\dfrac{\int_{\bn}|\nabla_{\bn} u|^2 \, \mathrm{~d} V_{\mathbb{B}^{N}} }{\int_{\bn} |u|^2 \, \mathrm{~d} V_{\mathbb{B}^{N}}}. 
\end{equation}

\begin{rem}
A  consequence of \eqref{firsteigen} is that if $\lambda < \frac{(N-1)^2}{4},$ then

$$
||u||_{H_{\lambda}} := ||u||_{\lambda} := \left[ \int_{\bn} \left( |\nabla_{\bn} u|^2 - \lambda \, u^2 \right) \, \mathrm{~d} V_{\mathbb{B}^{N}} \right]^{\frac{1}{2}}, \quad u \in C_c^{\infty}(\bn)
$$
is a norm, equivalent to the $H^1(\bn)$ norm and the corresponding inner product is given by $\langle u, v\rangle_{H_{\lambda}}.$

\end{rem}


\section{Palais-Smale characterization}\label{SecPS}

\noi In this section, we study the Palais-Smale sequences (PS sequences) corresponding to the problem $(P_\la)$.
To be precise, define the associated energy functional $I_{\lambda, a}$ as
\begin{equation*}
I_{\lambda,a}(u)=\frac{1}{2} \int_{\mathbb{B}^{N}}\left[\left|\nabla_{\mathbb{B}^{N}} u\right|_{\mathbb{B}^{N}}^{2}-\lambda u^{2}\right] \mathrm{~d} V_{\mathbb{B}^{N}}-\frac{1}{p+1} \int_{\mathbb{B}^{N}}a(x)|u|^{p+1} \mathrm{~d} V_{\mathbb{B}^{N}}. 
\end{equation*}
We say a sequence $u_{n} \in H^{1}\left(\mathbb{B}^{N}\right)$ is a Palais-Smale sequence for $I_{\lambda,a}$ at a level $d$ if $I_{\lambda,a}\left(u_{n}\right) \rightarrow d$ and $I_{\lambda,a}^{\prime}\left(u_{n}\right) \rightarrow 0$ in $H^{-1}\left(\mathbb{B}^{N}\right) .$ One can easily see that PS sequences are bounded. 
Throughout this section, we assume $a(x) \rightarrow 1$ as $d(x,0) \rightarrow \infty$.

\medskip

\subsection{Compactness and Palais-Smale sequences} Before introducing the Palais-Smale characterization, we shall remark on the compactness properties of $(P_\la)$ posed in the hyperbolic space. To see this, let $u \in H^{1}\left(\mathbb{B}^{N}\right)$ and 
$b_n \in \bn$ such that $b_n \rightarrow \infty$ and let $\tau_n$ be the hyperbolic translation given by \eqref{hypt} such that $\tau_n(0) = b_n.$ Define $u_n = u \circ \tau_n,$
then $||u_n||_{H^{1}\left(\mathbb{B}^{N}\right)} = ||u||_{H^{1}\left(\mathbb{B}^{N}\right)}$ but $u_n \rightharpoonup 0$ weakly in $H^1(\bn).$ This shows that the embedding $H^1(\bn) \hookrightarrow L^p(\bn)$ is not compact for any $2 < p < \frac{2N}{N-2}$ even in the 
subcritical case. 

\medskip 

\noi It is worth noticing that if $v$ is a solution of $(P_\la)$ with $a(x) \equiv 1,$ then $v\circ \tau$ is also a solution for any $\tau \in I(\bn),$ where $I(\bn)$ is an isometry group.  In particular, if we define 

$$
v_n = v\circ \tau_n,
$$
where $\tau_n $ as defined in \eqref{hypt} with $\tau_n (0) \rightarrow \infty,$ then $v_n$ is a PS sequence converging weakly to zero. 
Thus in the limiting case, i.e., $a(x) \rightarrow 1,$ as $d(x, 0) \rightarrow \infty,$ the functional $I_{\lambda, a}$ for the Palais-Smale sequences 
exhibits sequences of the form $v_n.$ To be precise, we state the following proposition.

\medskip

\begin{prop}\label{prop1}
Let $\left\{u_{n}\right\} \subset H^{1}\left(\mathbb{B}^{N}\right)$ be a PS sequence for $I_{\lambda,a}$ at a level $d\geq 0.$ Then there exists a subsequence (still denoted by $\left\{u_{n}\right\}$) for which the following holds:
there exists an integer $m \geq 0$, sequences $\tau_{n}^{i}$ $\in I\left( \mathbb{B}^{N}\right)$, functions $\bar{u}, w_{i}\in H^1(\bn)$ for $1 \leq i \leq m$ such that
\begin{equation*}
-\Delta_{\mathbb{B}^{N}} \bar{u}-\lambda \bar{u}=a(x)|\bar{u}|^{p-1} \bar{u} \; \text{in}\; H^{-1}(\bn),
\end{equation*}

\begin{equation*}
\begin{gathered}
-\Delta_{\mathbb{B}^{N}} w_{i}-\lambda w_{i} = |w_{i}|^{p-1}w_i\; \text{in}\; H^{-1}(\bn), w_{i} \not \equiv 0, 
\end{gathered}
\end{equation*}
\begin{equation*}
\begin{gathered}
u_{n}-\left(\bar{u}+\sum_{i=1}^{m} w_{i}\left(\tau_{n}^{i}(\bullet)\right)\right) \rightarrow 0 \text { as } n \rightarrow \infty, \\
I_{\lambda,a}\left(u_{n}\right) \rightarrow I_{\lambda,a}(\bar{u})+\sum_{i=1}^{m} I_{\lambda,1}\left(w_{i}\right) \text { as } n \rightarrow \infty, \\
\tau_{n}^{i}(0) \rightarrow \infty,\; d(\tau_{n}^{i}(0),\tau_{n}^{j}(0))\rightarrow \infty \text { as } n \rightarrow \infty, \text { for } 1 \leq i \neq j \leq m. 
\end{gathered}
\end{equation*}
\end{prop}
To prove the above proposition, we need the following auxiliary lemmas.
\begin{lem}\label{lem2}
Let $\left\{u_{n}\right\}$ be a bounded sequence in $H^{1}\left(\mathbb{B}^{N}\right)$ such that $$\sup_{y \in \mathbb{B}^N}\int_{B(y, r)}|u_{n}|^q \mathrm{~d} V_{\mathbb{B}^{N}}\rightarrow 0 \quad \text{as} \quad n \rightarrow \infty,$$
for some $r>0$ and $2\leq q<2^{*}$. Then $u_{n} \rightarrow 0$ strongly in $L^{p}\left(\mathbb{B}^{N}\right)$ for all $p \in\left(2,2^{*}\right)$.
 In addition, if $u_{n}$ satisfies
\begin{equation}
-\Delta_{\mathbb{B}^{N}} u_{n}-\lambda u_{n}-a(x)|u_{n}|^{p-1} u_{n} \rightarrow 0 \quad \text { in } \quad H^{-1}\left(\mathbb{B}^{N}\right), \label{2bb}
\end{equation}
then $u_{n} \rightarrow 0$ strongly in $H^{1}\left(\mathbb{B}^{N}\right)$.
\begin{proof}
 For any $s \in\left(q, 2^{*}\right),$ we have from the interpolation inequality
\begin{equation}
\|u_{n}\|_{L^{s}(B(y, r))} \leq\|u_{n}\|_{L^{q}(B(y, r))}^{1-\lambda}\|u_{n}\|_{L^{2^{*}}(B(y, r))}^{\lambda}, \label{2e}
\end{equation}
 where $\lambda=\frac{s-q}{2^{*}-q} \frac{2^{*}}{s}.$ The Sobolev inequalities on geodesic balls in the hyperbolic space imply that there exists a positive constant $C>0$ independent of $y \in \bn$ such that 
 
 \begin{equation}\label{2f}
 \int_{B(y, r)}|u_{n}|^{2^{*}} d V_{\mathbb{B}^{N}}
  \leq  C \, \left(\int_{B(y, r)}\left(|u_{n}|^{2}+\left|\nabla_{\mathbb{B}^{N}} u_{n}\right|^{2}\right) \mathrm{~d} V_{\mathbb{B}^{N}}\right)^{\frac{2^{*}}{2}}.
 \end{equation}
 
\medskip 

\noi Using \eqref{2e} and \eqref{2f}, we get
\begin{equation*}
\int_{B(y, r)}|u_{n}|^{s} \mathrm{~d} V_{\mathbb{B}^{N}} \leq C|| u_{n}||_{L^{q}(B(y, r))}^{s(1-\lambda)} \left(\int_{B(y, r)}\left(|u_{n}|^{2}+\left|\nabla_{\mathbb{B}^{N}} u\right|^{2}\right) \mathrm{~d} V_{\mathbb{B}^{N}}\right)^{\frac{\lambda s}{2}}.
\end{equation*}
Now covering $\mathbb{B}^{N}$ by balls of radius $r$, in such a way that each point of $\mathbb{B}^{N}$ is contained inside at most $N_0$ balls, we find\\
\begin{equation*}
\int_{\mathbb{B}^{N}}|u_{n}|^{s} \mathrm{~d} V_{\mathbb{B}^{N}} \leq C \sup _{y \in \mathbb{B}^{N}}\|u_{n}\|_{L^{q}(B(y, r))}^{s(1-\lambda)} \int_{\mathbb{B}^{N}}\left(|u_{n}|^{2}+\left|\nabla_{\mathbb{B}^{N}} u_{n}\right|^{2}\right) \mathrm{~d} V_{\mathbb{B}^{N}}.
\end{equation*}
Therefore, utilizing the lemma hypothesis, we get for the sequence $\left\{u_{n}\right\}$ that $\left\|u_{n}\right\|_{L^{s}\left(\mathbb{B}^{N}\right)} \rightarrow 0$ for all $s \in\left(q, 2^{*}\right)$. This proves the first part of the lemma if $q=2$; otherwise, if $q>2$, then again, one can argue similarly by choosing $s \in(2, q)$.
In addition, if \eqref{2bb} is satisfied, then we obtain\\
\begin{equation}
\left|\left\langle-\Delta_{\mathbb{B}^{N}} u_{n}-\lambda u_{n}-a(x)|u_{n}|^{p-1} u_{n}, u_{n}\right\rangle_{H_{\lambda}} \right|=o(1)\left\|u_{n}\right\|_\la. \label{2cc}
\end{equation}
Since $\left\{u_{n}\right\}$ is bounded in $H^{1}\left(\mathbb{B}^{N}\right)$, the RHS is $o(1) .$ On the other hand, for the LHS we notice that because $u_{n}$ is bounded in $H^{1}\left(\mathbb{B}^{N}\right)$ and $u_{n} \rightarrow 0$ strongly in $L^{r}\left(\mathbb{B}^{N}\right)$, for $r \in\left(2,2^{*}\right)$, we must have $u_{n} \rightharpoonup 0$ weakly in $H^{1}\left(\mathbb{B}^{N}\right)$. Moreover, by the previous part, $u_{n} \rightarrow 0$ strongly in $L^{p+1}\left(\mathbb{B}^{N}\right) .$ Thus by \eqref{2cc}, we get $u_{n} \rightarrow 0$ strongly in $H^{1}\left(\mathbb{B}^{N}\right).$ 
\end{proof}
\end{lem}


\begin{lem}\label{lem3}
Let $\phi_{k} \rightharpoonup \phi$ weakly in $H^{1}\left(\mathbb{B}^{N}\right)$, then we have
\begin{equation*}
a\left|\phi_{k}\right|^{p-1} \phi_{k}-a|\phi|^{p-1} \phi \rightarrow 0 \quad \text {in } \quad H^{-1}\left(\mathbb{B}^{N}\right).
\end{equation*}
\begin{proof}
Defining $\psi_{k} = \phi_{k}-\phi$,\;gives $\psi_{k} \rightharpoonup 0$ weakly in $H^{1}\left(\mathbb{B}^{N}\right).$ Particularly, $\left\{\psi_{k}\right\}$ is bounded in $H^{1}\left(\mathbb{B}^{N}\right) .$ Thus, up to a subsequence, $\psi_{k} \rightarrow 0$ strongly in $L_{\text {loc }}^{q}(\mathbb{B}^{N})$ for all $2<q<2^{*}$ and $\psi_{k} \rightarrow 0$ a.e.. As a result, $a\left|\phi+\psi_{k}\right|^{p-1}\left(\phi+\psi_{k}\right)-a|\phi|^{p-1} \phi \rightarrow 0$ a.e.. Therefore, it follows from Vitaly's convergence theorem that $a\left|\phi+\psi_{k}\right|^{p-1}\left(\phi+\psi_{k}\right)-a|\phi|^{p-1} \phi \rightarrow 0$ strongly in $L_{\text {loc }}^{\frac{p+1}{p}}\left(\mathbb{B}^{N}\right) .$ Also, we can see that for every $\varepsilon>0$, there exists $K_{\varepsilon}>0$ such that\\
\begin{equation}
\left|a | \phi+\psi_{k}|^{p-1}\left(\phi+\psi_{k}\right)- a|\phi|^{p-1} \phi\right|^{\frac{p+1}{p}} \leq \varepsilon\left|\psi_{k}\right|^{p+1}+K_{\varepsilon}|\phi|^{p+1}. \label{2g}
\end{equation}
Furthermore, as $\psi_{k} \rightarrow 0$ weakly in $H^{1}\left(\mathbb{B}^{N}\right)$ implies $\psi_{k}$ is uniformly bounded in $L^{p+1}\left(\mathbb{B}^{N}\right)$ and the fact that $|\phi|^{p+1} \in L^{1}\left(\mathbb{B}^{N}\right)$, it can be immediately observed from \eqref{2g} that given $\varepsilon>0$, there exists $R>0$ for which\\
$$\int_{\mathbb{R}^{N} \backslash B(0, R)}|a| \phi+\left.\psi_{k}\right|^{p-1}\left(\phi+\psi_{k}\right)-\left.a|\phi|^{p-1} \phi\right|^{\frac{p+1}{p}} \mathrm{~d} V_{\mathbb{B}^{N}}<\varepsilon.$$
Consequently, $a\left|\phi+\psi_{k}\right|^{p-1}\left(\phi+\psi_{k}\right)-a|\phi|^{p-1} \phi \rightarrow 0$ strongly in $L^{\frac{p+1}{p}}\left(\mathbb{B}^{N}\right) .$ Hence as $H^{1}\left(\mathbb{B}^{N}\right)$ is continuously embedded in $L^{p+1}\left(\mathbb{B}^{N}\right)$, which is the dual space of $L^{\frac{p+1}{p}}\left(\mathbb{B}^{N}\right)$, it follows that $a\left|\phi+\psi_{k}\right|^{p-1}\left(\phi+\psi_{k}\right)-a|\phi|^{p-1} \phi \rightarrow 0$ strongly in $H^{-1}\left(\mathbb{B}^{N}\right).$
\end{proof}
\end{lem}

\medskip

\begin{lem}\label{lem4}
For each $c_{0} \geq 0$, there exists $\delta>0$ such that if $v \in H^{1}\left(\mathbb{B}^{N}\right)$ solves
\begin{equation*}
-\Delta_{\mathbb{B}^{N}} v-\lambda v=|v|^{p-1} v \text { in } H^{-1}(\mathbb{B}^{N}),
\end{equation*}
and $\|v\|_{H^{1}\left(\mathbb{B}^{N}\right)} \leq c_{0},\|v\|_{L^{2}\left(\mathbb{B}^{N}\right)} \leq \delta$, then $v \equiv 0$.
\begin{proof}Taking $v$ as a test function yields
\begin{equation*}
C^{\prime}\, \|v\|_{H^1(\bn)}^{2}  \leq \|v\|_{\lambda}^2 \, = \, \int_{\mathbb{B}^{N}}|v|^{p+1} \mathrm{~d}V_{\mathbb{B}^{N}} \leq \|v\|_{L^{2}\left(\mathbb{B}^{N}\right)}^{\alpha}\left(\int_{\mathbb{B}^{N}}|v|^{\frac{2 N}{N-2}} \mathrm{~d} V_{\mathbb{B}^{N}}\right)^{\beta} 
\leq C  \delta^{\alpha}\|v\|_{H^1(\bn) }^{\gamma}, 
\end{equation*}
with  $\gamma=\frac{2 N}{N-2} \beta,$
by Hölder and Sobolev inequalities, where $C^{\prime}, C$ denote non-negative constants independent of $c_{0}, \delta$ and $\alpha=\left(\frac{2 N}{N-2}-(p+1)\right)\left(\frac{2 N}{N-2}-2\right)^{-1},\; \beta=(p-1)\left(\frac{2 N}{N-2}-2\right)^{-1}$.\\
Now, if $p \geq 1+\frac{4}{N}$ then $\gamma \geq 2$ and we conclude easily if $\delta$ is small enough. On the other hand if $p<1+\frac{4}{N}$, we deduce
\begin{equation*}
\|v\|_{H^{1}\left(\mathbb{B}^{N}\right)} \leq C \delta^{k} \text { for } k=\alpha(2-\gamma)^{-1}> 0.
\end{equation*}
 Therefore, choosing $\delta>0$ small enough, we can conclude the lemma.
\end{proof}
\end{lem}

\noi {\bf Proof of Proposition \ref{prop1}:}

The boundedness of a PS sequence in $H^{1}\left(\mathbb{B}^{N}\right)$ follows from the standard arguments. To be precise, $I_{\lambda,a}\left(u_{n}\right)=d+o(1)$ and $I_{\lambda,a}^{\prime}\left(u_{n}\right) (u_{n})=o(1)\left\|u_{n}\right\|_{H^{1}\left(\mathbb{B}^{N}\right)}$, evaluating $I_{\lambda,a}\left(u_{n}\right)-\frac{1}{p+1}I_{\lambda,a}^{\prime}\left(u_{n}\right) (u_{n})$, we obtain
\begin{equation*}
\left\|u_{n}\right\|_{H^{1}\left(\mathbb{B}^{N}\right)}^{2} \leq C+o(1)\left\|u_{n}\right\|_{H^{1}\left(\mathbb{B}^{N}\right)},
\end{equation*}
and hence boundedness follows. As a result, we can assume $u_{n} \rightharpoonup u$ weakly in $H^{1}\left(\mathbb{B}^{N}\right)$ up to a subsequence.
Moreover, $ |I_{\lambda,a}^{\prime}(u_{n})(v)| \rightarrow 0$ as $n \rightarrow$ $\infty \;\; \forall v \in H^{1}\left(\mathbb{B}^{N}\right)$ implies 

\begin{equation}
\int_{\mathbb{B}^{N}} \nabla_{\mathbb{B}^{N}} u_{n} \nabla_{\mathbb{B}^{N}} v \mathrm{~d} V_{\mathbb{B}^{N}} \, -\, \lambda \int_{\mathbb{B}^{N}} u_{n} v \mathrm{~d} V_{\mathbb{B}^{N}}-\int_{\mathbb{B}^{N}} a(x)\left|u_{n}\right|^{p-1} u_{n} v \mathrm{~d} V_{\mathbb{B}^{N}} \rightarrow 0, \label{2i}
\end{equation}
as $n \rightarrow \infty$ for all $v \in H^{1}\left(\mathbb{B}^{N}\right)$.
 
Furthermore, using $\ref{lem3},$ we deduce that weak limit $u$ satisfies
 \begin{equation*}
-\Delta_{\mathbb{B}^{N}} u-\lambda u=a(x)|u|^{p-1} u \text{ in } H^{-1}(\mathbb{B}^{N}).
\end{equation*}

\noindent Next, we show that $u_{n}\, - \, u$ is a PS sequence for $I_{\lambda,a}$ at the level {$d-I_{\lambda,a}(u).$}
First of all, applying Brezis-Lieb lemma, we obtain the following equations\\
\begin{align*}
&\int_{\mathbb{B}^{N}} \left| \nabla_{\mathbb{B}^{N}} u_{n} \right| ^ {2}\mathrm{~d} V_{\mathbb{B}^{N}} - \int_{\mathbb{B}^{N}} \left| \nabla_{\mathbb{B}^{N}} u \right| ^ {2}\mathrm{~d} V_{\mathbb{B}^{N}}
=\int_{\mathbb{B}^{N}} \left| \nabla_{\mathbb{B}^{N}} \left(u_{n}-u\right) \right| ^ {2}\mathrm{~d} V_{\mathbb{B}^{N}} +o(1), \\
&\int_{\mathbb{B}^{N}}\left|u_{n}\right|^{2} \mathrm{~d} V_{\mathbb{B}^{N}} -\int_{\mathbb{B}^{N}}|u|^{2} \mathrm{~d} V_{\mathbb{B}^{N}}=\int_{\mathbb{B}^{N}}\left|u_{n}-u\right|^{2} \mathrm{~d} V_{\mathbb{B}^{N}}+o(1),
 \end{align*}
 \begin{align*}
\int_{\mathbb{B}^{N}} a(x)\left|u_{n}\right|^{p+1} \mathrm{~d} V_{\mathbb{B}^{N}}(x) &-\int_{\mathbb{B}^{N}} a(x)|u|^{p+1} \mathrm{~d} V_{\mathbb{B}^{N}}(x) \\
&=\int_{\mathbb{B}^{N}} a(x)\left|u_{n}-u\right|^{p+1} \mathrm{~d} V_{\mathbb{B}^{N}}(x)+o(1). 
\end{align*}
Thus applying the above equations, we obtain
\begin{equation*}
\begin{aligned}
I_{\lambda,a}\left(u_{n}-u\right) =&\frac{1}{2}\left(\left\|u_{n}\right\|_{\la}^{2}-\|u\|_{\la}^{2}\right)\\
 &-\frac{1}{p+1}\left(\int_{\mathbb{B}^{N}} a(x)\left|u_{n}\right|^{p+1} \mathrm{~d} V_{\mathbb{B}^{N}}(x) -\int_{\mathbb{B}^{N}} a(x)|u|^{p+1} \mathrm{~d} V_{\mathbb{B}^{N}}(x)\right)+o(1) \\
 &\longrightarrow  {d-I_{\lambda,a}(u)} \quad \text { as } n \rightarrow \infty. \\
\end{aligned}
\end{equation*}
 Now for $\phi \in  H^{1}\left(\mathbb{B}^{N}\right)$, 
\begin{align*}
    I_{\lambda,a}^{\prime}(u_{n}-u) (\phi)
=&\int_{\mathbb{B}^{N}} \nabla_{\mathbb{B}^{N}} \left(u_{n}-u\right) \nabla_{\mathbb{B}^{N}}\phi \, \mathrm{~d} V_{\mathbb{B}^{N}} \\
&-\lambda \int_{\mathbb{B}^{N}} \left(u_{n}-u\right) \phi \mathrm{~d} V_{\mathbb{B}^{N}} - \int_{\mathbb{B}^{N}} a(x)\left|u_{n}-u\right|^{p-1} \left(u_{n}-u\right) \phi \mathrm{~d} V_{\mathbb{B}^{N}}.\\
\end{align*}
{\text{Using \eqref{2i}, we get}
\begin{equation*}
  {I_{\lambda,a}^{\prime}(u_{n}-u), (\phi)}=\int_{\mathbb{B}^{N}} a(x)\left[\left|u_{n}\right|^{p-1} u_{n}- \left|u\right|^{p-1} u- \left|u_{n} \, - \, u\right|^{p-1} \left(u_{n} \, - \, u\right) \right] \phi  \mathrm{~d} V_{\mathbb{B}^{N}} + o(1).\\
\end{equation*}}

\noi \text{Applying the Hölder inequality, the above quantity can be estimated as}\\
\begin{equation*}
\left|\phi\right|_{L^{2^{*}}\left(\mathbb{B}^{N}\right)}\left[\int_{\mathbb{B}^{N}}\left| a(x)\left(\left|u_{n}\right|^{p-1} u_{n}-|u|^{p-1} u-\left|u_{n}-u\right|^{p-1}\left(u_{n}-u\right)\right)\right|^{\frac{2 N}{N+2}} \mathrm{~d} V_{\mathbb{B}^{N}}\right]^{\frac{N+2}{2 N}}.\\
\end{equation*}
Now, we can observe that the term inside the bracket is of $o(1)$, which follows from the standard arguments using Vitali's convergence theorem.
Hence we get $I_{\lambda,a}^{\prime}\left(u_{n}-u\right)= o(1)$.

 \noi {Therefore, in view of the Lemma $\ref{lem2}$ we have, either $u_{n}-u \rightarrow 0$ strongly} in $H^{1}\left(\mathbb{B}^{N}\right)$, in that case we are done or there exists $\alpha>0$, such that up to a subsequence
\begin{equation*}
Q_{n}(1):=\sup _{y \in \mathbb{B}^{N}} \int_{B(y, 1)}\left|u_{n}-u\right|^{2} \mathrm{~d} V_{\mathbb{B}^{N}}>\alpha>0.
\end{equation*}
Consequently, we can find a sequence $\left\{y_{n}\right\} \subset \mathbb{B}^{N}$ such that
\begin{equation}
\int_{B\left(y_{n}, 1\right)}\left|u_{n}-u\right|^{2} \mathrm{~d} V_{\mathbb{B}^{N}}\geq \alpha . \label{2n}
\end{equation}
Now define $v_{n}(x):=\left(u_{n}-u\right)\left(T_{n}(x)\right)$, where $T_{n}(x)=\tau_{y_{n}}(x)$  and $\tau_{y_{n}}$ is the hyperbolic translation of $\mathbb{B}^{N}$ by $y_{n}$. Now, as $\tau_{y_{n}}$ is an isometry, so $v_{n}$ will form a PS sequence at the same level as $u_{n}-u$ which implies ${v}_{n}$ is also bounded in $H^{1}\left(\mathbb{B}^{N}\right)$, and hence converges weakly in $H^{1}\left(\mathbb{B}^{N}\right)$ upto a subsequence to say $v$.
The compact embedding of  $H^{1}\left(B\left(0, 1\right)\right) $ into $ L^{2}\left(B\left(0, 1\right)\right)$  and \eqref{2n} implies $v\not\equiv 0$. 
Also, as $u_{n}-u \rightharpoonup 0$ weakly in $H^{1}\left(\mathbb{B}^{N}\right)$ and using \eqref{2n} it follows that
\begin{equation*}
\tau_{y_{n}}(0) \rightarrow \infty \quad \text { as } n \rightarrow \infty.
\end{equation*}
Let us now define,
\begin{equation*}
w_{n}:= v_{n}-v .
\end{equation*}
Clearly, the fact that $v_{n} \rightharpoonup v$ implies $w_{n} \rightharpoonup 0$ weakly in $H^{1}\left(\mathbb{B}^{N}\right).$ Applying this and Lemma \ref{lem3}, in the definition of $I_{\lambda,1}^{\prime}\left(w_{n}\right)$ results in
\begin{equation*}
I_{\lambda,1}^{\prime}\left(w_{n}\right)=o(1) \text { in } H^{-1}\left(\mathbb{B}^{N}\right), 
\end{equation*}
i.e.,
\begin{equation*}
-\Delta_{\mathbb{B}^{N}} w_{n}-\lambda w_{n}-|w_{n}|^{p-1} w_{n} \rightarrow 0 \quad \text { in } H^{-1}\left(\mathbb{B}^{N}\right).
\end{equation*}
 Next, we show that $v$ satisfies 
\begin{equation}
-\Delta_{\mathbb{B}^{N}} v-\lambda v = |v|^{p-1} v  \quad \text { in } H^{-1}\left(\mathbb{B}^{N}\right), \quad v \in H^{1}\left(\mathbb{B}^{N}\right). \label{2p}
\end{equation}
To prove this, let $\tilde{v} \in C_{0}^{\infty}\left(\mathbb{B}^{N}\right)$. Since, $v_{n} \rightharpoonup v$ weakly in $H^1(\bn)$, we estimate  as follows:
\begin{equation*}
\begin{aligned}
\langle v, \tilde{v} \rangle_{H_{\lambda}} &= \lim _{n \rightarrow \infty}\left\langle v_{n},\tilde{v}\right\rangle_{H_{\lambda}} \\
&= \lim _{n \rightarrow \infty} \int_{\mathbb{B}^{N}} \nabla_{\mathbb{B}^{N}} v_{n} \nabla_{\mathbb{B}^{N}} \tilde{v}  \mathrm{~d} V_{\mathbb{B}^{N}} -\lambda\int_{\mathbb{B}^{N}} v_{n} \tilde{v}  \mathrm{~d} V_{\mathbb{B}^{N}}\\
&= \lim _{n \rightarrow \infty}\int_{\mathbb{B}^{N}} \nabla_{\mathbb{B}^{N}} \left(u_{n}-u\right)(T_{n}(x)) \nabla_{\mathbb{B}^{N}} \tilde{v}  \mathrm{~d} V_{\mathbb{B}^{N}} -\lambda\int_{\mathbb{B}^{N}} \left(u_{n}-u\right)\left(T_{n}(x)\right) \tilde{v}  \mathrm{~d} V_{\mathbb{B}^{N}}\\
&= \lim _{n \rightarrow \infty} \int_{\mathbb{B}^{N}} a(y)\left|\left(u_{n}-u\right)(y)\right|^{p-1}\left(u_{n}-u\right)(y) \tilde{v}\left({T_{n}}^{-1}(y)\right)\mathrm{~d} V_{\mathbb{B}^{N}} \\
&= \lim _{n \rightarrow \infty} \int_{\mathbb{B}^{N}} a\left(T_{n}(x)\right)\left|v_{n}(x)\right|^{p-1} v_{n}(x) \tilde{v}(x) \mathrm{~d} V_{\mathbb{B}^{N}}.
\end{aligned}
\end{equation*}
Also,
%
%
\begin{equation*}
\begin{aligned} 
&\left|\int_{\mathbb{B}^{N}}  a\left(T_{n}(x)\right)\left|v_{n}(x)\right|^{p-1} v_{n}(x) \tilde{v}(x) \mathrm{~d} V_{\mathbb{B}^{N}}-\int_{\mathbb{B}^{N}}|v|^{p-1} v \tilde{v}  \mathrm{~d} V_{\mathbb{B}^{N}}\right| \\
 &\leq \left|\int_{\mathbb{B}^{N}} a\left(T_{n}(x)\right)\left(\left|v_{n}\right|^{p-1} v_{n}-\left|v\right|^{p-1} v \right) \tilde{v} \mathrm{~d} V_{\mathbb{B}^{N}}\right| \\
&\;\; +\left|\int_{\mathbb{B}^{N}}\left( a\left(T_{n}(x)\right)-1\right)\left| v\right|^{p-1} v \tilde{v} \mathrm{~d} V_{\mathbb{B}^{N}}\right|\\ 
&:=I_{n}^{1}+J_{n}^{1}. 
\end{aligned}
\end{equation*}

\noi Since $T_{n}(0) \rightarrow \infty,\left|v\right|^{p-1} v \tilde{v} \in L^{1}\left(\mathbb{B}^{N}\right), a \in L^{\infty}\left(\mathbb{B}^{N}\right)$ and $a(x) \rightarrow 1$ as $d(x,0) \rightarrow \infty$, the dominated convergence theorem yields
\begin{equation*}
\lim _{n \rightarrow \infty} J_{n}^{1}=0 .
\end{equation*}
On the other hand, since $\tilde{v}$ has a compact support; applying Vitaly's convergence theorem gives 
\begin{equation*}
\lim _{n \rightarrow \infty} I_{n}^{1} \leq \lim _{n \rightarrow \infty}\|a\|_{L^{\infty}\left(\mathbb{B}^{N}\right)} \int_{\text {supp } v}\left|\left|v_{n}\right|^{p-1} v_{n}-\left|v\right|^{p-1} v \right| \left|\tilde{v}\right| \mathrm{~d} V_{\mathbb{B}^{N}}=0 .
\end{equation*}
The above two estimates  help us to conclude that $v$ satisfies \eqref{2p}.

\noi Further, applying Brezis-Lieb Lemma, we get
\begin{equation*}
\begin{gathered}
\int_{\mathbb{B}^{N}} \left| \nabla_{\mathbb{B}^{N}} v_{n} \right| ^ {2}\mathrm{~d} V_{\mathbb{B}^{N}} - \int_{\mathbb{B}^{N}} \left| \nabla_{\mathbb{B}^{N}} v \right| ^ {2}\mathrm{~d} V_{\mathbb{B}^{N}}
-\int_{\mathbb{B}^{N}} \left| \nabla_{\mathbb{B}^{N}} w_{n} \right| ^ {2}\mathrm{~d} V_{\mathbb{B}^{N}} \rightarrow 0,\\ 
\int_{\mathbb{B}^{N}}\left|v_{n}\right|^{2} \mathrm{~d} V_{\mathbb{B}^{N}} -\int_{\mathbb{B}^{N}}|v|^{2} \mathrm{~d} V_{\mathbb{B}^{N}}-\int_{\mathbb{B}^{N}}\left|w_{n}\right|^{2} \mathrm{~d} V_{\mathbb{B}^{N}} \rightarrow 0,
\end{gathered}
 \end{equation*}
as $n \rightarrow \infty$.
In view of the above steps, we rerun the procedure for the Palais-Smale (PS) sequence $v_{n}-v$ to end up in either of the two cases. If it converges to zero, we stop, or else we repeat the steps. But it is worth noticing that this process has to terminate in finitely many steps, and we obtain $v_{1}, v_{2}, \ldots, v_{n}$ which denotes the limit solutions of \eqref{2p} obtained through the procedure, we have
\begin{equation*}
\sum_{i=1}^{n} \int_{\mathbb{B}^{N}}\left|v_{i}\right|^{2} \mathrm{~d} V_{\mathbb{B}^{N}} \leq \liminf _{n \rightarrow \infty} \int_{\mathbb{B}^{N}}\left|u_{n}-u\right|^{2} \mathrm{~d} V_{\mathbb{B}^{N}}.
\end{equation*}
Thus in view of Lemma \ref{lem4}, $n$ can not approach infinity.

\section{Key Lemma : Energy estimates involving hyperbolic bubbles}\label{keylemma}

\noi In this section, we derive key energy estimates of the functional associated with $(P_\la)$ involving hyperbolic bubbles.  To this end, let us recall the asymptotic estimates of positive solutions to the corresponding homogeneous problem
\begin{equation}
	\begin{gathered}
		-\Delta_{\mathbb{B}^{N}} w-\lambda w=|w|^{p-1} w, \; 	w>0\; \text { in } \mathbb{B}^{N}, 
		w \in H^{1}\left(\mathbb{B}^{N}\right). \label{3d}
	\end{gathered}
\end{equation}
Then by elliptic regularity, any solution, $w\in H^1(\bn),$ is also in $C^\infty$ and satisfies the decay property (See \cite[Lemma~3.4]{MS}): for every $\varepsilon > 0,$ there exist positive constants $C_1^{\varepsilon}$ and $C_2^{\varepsilon}$ such that there holds 
\begin{equation*}
	C_1^{\varepsilon} \e^{-(c(N,\lambda) + \varepsilon) \, d(x,0)} \leq w(x) \leq C_2^{\varepsilon} \e^{-(c(N,\lambda) - \varepsilon) \, d(x,0)}, \quad \hbox{for all} \ x \in \bn,
\end{equation*}
where $c(N, \lambda) \, =\, \frac{1}{2} (N-1+\sqrt{(N-1)^{2}-4 \lambda}).$ This decay property will play a pivotal role
in the subsequent analysis. 

\medskip

 \noi We define the functionals $J, J_\infty: H^1(\bn)\rightarrow \mathbb R$ as 
\begin{equation}
J(u):=\frac{\|u\|_{\lambda}^{2}}{\left(\int_{\mathbb{B}^{N}} a(x)|u(x)|^{p+1} \mathrm{~d} V_{\mathbb{B}^{N}}(x)\right)^{\frac{2}{p+1}}}, \quad J_{\infty}(u):=\frac{\|u\|_{\lambda}^{2}}{\left(\int_{\mathbb{B}^{N}}|u(x)|^{p+1} \mathrm{~d} V_{\mathbb{B}^{N}}(x)\right)^{\frac{2}{p+1}}} \\ \label{3b}
\end{equation}
and the energy levels
\begin{equation}
S_{1, \lambda}:=\inf _{u \in H^{1}\left(\mathbb{B}^{N}\right) \backslash\{0\}} J_{\infty}(u), \quad S_{m, \lambda}:=m^{\frac{p-1}{p+1}} S_{1, \lambda}, \quad m=2,3,4, \cdots \label{3c}
\end{equation}


\medskip

Let us recall a convex inequality. 
\begin{lem} \label{lem3b}
Let $p>1$ be any real number. Then for any non-negative real numbers $a,b$ there holds 
\begin{equation*}
(a+b)^{p+1} \geq a^{p+1}+b^{p+1}\, + \, p\left(a^{p} b+a b^{p}\right). 
\end{equation*}
\end{lem}

For the proof of the above lemma, we refer \cite{CP}.
\begin{lem} \label{lem3c}
Let $a(x)$ satisfies (a1)-(a2) and let $w$ be a unique radial solution of \eqref{3d}. Then, there exists a large number $R_{0}$, such that for any $R \geq R_{0},\; d(x_{1},0) \geq R^{\alpha},\; d(x_{2},0)\geq R^{\alpha},\; R^{\alpha^{\prime}} \leq d\left(x_{1},x_{2}\right) \leq  R^{\alpha^{\prime}-\alpha} \min \left\{d(x_{1},0), \;d(x_{2},0)\right\}$, where $\alpha>\alpha^{\prime}>1$, it holds
\begin{equation}
J\left(t u_{1}+(1-t) u_{2}\right)<S_{2, \lambda}, \label{3h}
\end{equation}
where $0 \leq t \leq 1$ and $ u_{i}=w\left(\tau_{-x_{i}}(\bullet)\right), i=1,2$.
\end{lem}
\begin{proof} \quad
\begin{enumerate}[label = \textbf{Step \arabic*:}]
\item 
In this step, we will prove $\eqref{3h}$ for $t=\frac{1}{2}$.  
Set,
\[A:=\|w\|_{H_{\lambda}}^{2}.\]
Then by taking  $w$ as a test function in $\eqref{3d}$, we obtain
\begin{equation*}
A =\int_{\mathbb{B}^{N}} w(x)^{p+1} \mathrm{~d} V_{\mathbb{B}^{N}}.
\end{equation*}
Therefore,
\begin{equation}
\begin{aligned}
	J\left(\frac{1}{2}(u_{1}+u_{2})\right) &=
J\left(u_{1}+u_{2}\right)\\ &=\frac{\left\|u_{1}\right\|_{H_{\lambda}}^{2}+\left\|u_{2}\right\|_{H_{\lambda}}^{2}+2\left\langle u_{1}, u_{2}\right\rangle_{H_{\lambda}}}{\left(\int_{\mathbb{B}^{N}}\left(u_{1}+u_{2}\right)^{p+1} \mathrm{~d} V_{\mathbb{B}^{N}}-\int_{\mathbb{B}^{N}}(1-a(x))\left(u_{1}+u_{2}\right)^{p+1} \mathrm{~d} V_{\mathbb{B}^{N}}\right)^{\frac{2}{p+1}}}  \\
& \leq \frac{ 2A+2\left\langle u_{1}, u_{2}\right\rangle_{{H_{\lambda}}}}{\left(\int_{\mathbb{B}^{N}}\left(u_{1}+u_{2}\right)^{p+1} \mathrm{~d} V_{\mathbb{B}^{N}}-\int_{\mathbb{B}^{N}}(1-a(x))_{+}\left(u_{1}+u_{2}\right)^{p+1} \mathrm{~d} V_{\mathbb{B}^{N}}\right)^{\frac{2}{p+1}}}. \label{3.g}
\end{aligned}
\end{equation}
Firstly, we estimate the integral, $\int_{\mathbb{B}^{N}} u_{1}^{p} \, u_{2} \, \mathrm{~d} V_{\mathbb{B}^{N}}.$ Since $w$ is positive, radially symmetric, symmetric decreasing and smooth, we get
\begin{equation*}
\begin{aligned}
\int_{\mathbb{B}^{N}} u_{1}^{p} u_{2} \mathrm{~d} V_{\mathbb{B}^{N}} &=\int_{\mathbb{B}^{N}} w\left(d(\tau_{-x_{1}}(x),0)\right)^{p} w\left(d(\tau_{-x_{2}}(x),0)\right) \mathrm{~d} V_{\mathbb{B}^{N}}(x)\\
& \geq \int_{d(x,x_{1}) \leq 1}w\left(d(\tau_{-x_{1}}(x),0)\right)^{p} w\left(d(\tau_{-x_{2}}(x),0)\right) \mathrm{~d} V_{\mathbb{B}^{N}}(x) \\
& \geq C \int_{d(x,x_{1})\leq 1} w\left(d(\tau_{-x_{2}}(x),0)\right)  \mathrm{~d} V_{\mathbb{B}^{N}}(x) \\
& =  \int_{d(x,x_{1})\leq 1} w\left(d(x,x_2)\right)  \mathrm{~d} V_{\mathbb{B}^{N}}(x).
\end{aligned}
\end{equation*}
 In the above calculations, we have used the following observations 
$$d(x,x_{1}) \leq 1 \Longrightarrow d(x,x_{2}) \leq d(x,x_{1})+d(x_{1},x_{2}) \leq 1+d(x_{1},x_{2}).$$ Consequently, $w\left(d(x,x_{2}\right) \geq w\left(1+d(x_{1},x_{2})\right)$. Therefore, for $d(x_{1},x_{2})>>1$,
\begin{equation*}
\int_{\mathbb{B}^{N}} u_{1}^{p} \, u_{2} \, \mathrm{~d} V_{\mathbb{B}^{N}} \geq C_{\varepsilon} w\left(1+d(x_{1},x_{2})\right) \geq C_{\varepsilon}
 \e^{-(c(N,\lambda) + \varepsilon)^{+} d(x_{1},x_{2})},
\end{equation*} 
where $C_{\varepsilon}$ is a positive constant independent of $R$, and for $\gamma \in \mathbb{R},$ $\gamma^{+}\left(\right.$respectively $\left.\gamma^{-}\right)$ stands for any $\gamma+\delta$ (respectively $\gamma-\delta$ ) with $\delta>0.$

\medskip

\noi Next, we estimate
\begin{align*}
&\int_{\mathbb{B}^{N}}(1-a)_{+}\left(u_{1}+u_{2}\right)^{p+1} \mathrm{~d} V_{\mathbb{B}^{N}} 
\leq C \sum_{i=1}^{2} \int_{\mathbb{B}^{N}}(1-a)_{+} u_{i}^{p+1} \mathrm{~d} V_{\mathbb{B}^{N}}\notag\\
&\leq C \sum_{i=1}^{2} \int_{\mathbb{B}^{N}} \frac{w\left(d(\tau_{-x_{i}}(x),0)\right)^{p+1}}{\e^{\delta d(x,0)}} \mathrm{~d} V_{\mathbb{B}^{N}}\notag\\
&= C \sum_{i=1}^{2} \left( \underbrace{ \int_{d(x,x_{i})> d(x_{i},0)} \frac{w\left(d(\tau_{-x_{i}}(x),0)\right)^{p+1}}{\e^{\delta d(x,0)}} \mathrm{~d} V_{\mathbb{B}^{N}}}_{I_1^i} 
+ \underbrace{\int_{d(x,x_{i}) \leq d(x_{i},0)} \frac{w\left(d(\tau_{-x_{i}}(x),0)\right)^{p+1}}{\e^{\delta d(x,0)}} \mathrm{~d} V_{\mathbb{B}^{N}}}_{I_2^i} \right).\notag\\
\end {align*}
Now we shall  estimate $I_{1}^{i}+ I_{2}^{i}$  with  $ \delta \in \left(K c(N, \lambda)+(N-1), (p+1) c(N, \lambda)\right),$ and choose $\varepsilon > 0$ such that 
$\delta < (p+1) (c(N, \lambda) - \varepsilon),$ where $K$ is fixed, and is such that $0<K< (p+1)-\frac{N-1}{c(N, \lambda)}.$ Before moving further, we  note that  $\frac{c(N, \lambda)}{N-1}> \frac{1}{2}$ and $p+1>2$ implies $(p+1)-\frac{N-1}{c(N, \lambda)}>0$. Let us first consider $I_{1}^{i}:$

\medskip

\begin{align*}
I_{1}^{i}&= \int_{d(x,x_{i})> d(x_{i},0)} \frac{w\left(d(\tau_{-x_{i}}(x),0)\right)^{p+1}}{\e^{\delta d(x,0)}} \mathrm{~d} V_{\mathbb{B}^{N}}(x)\notag\\
& \leq \int_{d(x,x_{i})> d(x_{i},0)} \frac{w\left(d(x_{i},0)\right)^{p+1}}{\e^{\delta d(x,0)}} \mathrm{~d} V_{\mathbb{B}^{N}}(x)\notag \\
&\leq C_{\varepsilon}\e^{-(c(N,\lambda) - \varepsilon)d(x_{i},0)(p+1)} \int_{d(x,x_{i})> d(x_{i},0)} \e^{-\delta d(x,0)} \mathrm{~d} V_{\mathbb{B}^{N}}(x)\notag\\
&\leq C_{\varepsilon}\e^{-(c(N,\lambda) - \varepsilon)d(x_{i},0)(p+1)} \int_{d(x,0)>0} \e^{-\delta d(x,0)} \mathrm{~d} V_{\mathbb{B}^{N}}(x)\notag\\
& = C_{\varepsilon}\e^{-(c(N,\lambda) - \varepsilon)d(x_{i},0)(p+1)} \int_{r=0}^{\infty} \e^{-\delta r} \e^{(N-1)r} \mathrm{~d} r\\
& \leq C_{\varepsilon}\e^{-(c(N,\lambda) - \varepsilon)d(x_{i},0)(p+1)},
\end{align*}
where passing from the second inequality to third we have used $d(x,x_{i})> d(x_{i},0)$ which in turn implies,  $d(x,0) \geq d(x, x_{i})- d(x_{i},0) > 0 $, and the last step follows from the choice of $\delta$ as $\delta > K c(N, \lambda)+(N-1) > N-1.$

\medskip

\noi Now, we estimate $I_{2}^{i}$ for $ \delta < (p+1) (c(N, \lambda) - \varepsilon),$ and using triangle inequality $d(x_{i},0) \leq d(x_{i},x)+d(x,0),$
\begin{align*}
I_{2}^{i}&= \int_{d(x,x_{i}) \leq d(x_{i},0)} \frac{w\left(d(\tau_{-x_{i}}(x),0)\right)^{p+1}}{\e^{\delta d(x,0)}} \mathrm{~d} V_{\mathbb{B}^{N}}(x)\\
&\leq C_{\varepsilon} \int_{d(x,x_{i}) \leq d(x_{i},0)} \e^{- (c(N,\lambda) - \varepsilon)d(x,x_{i})(p+1) } \e^{-\delta d(x,0))}\mathrm{~d} V_{\mathbb{B}^{N}}(x)\\
&\leq C_{\varepsilon} \int_{d(x,x_{i}) \leq d(x_{i},0)} \e^{- \delta d(x,x_{i})} \e^{-\delta d(x,0))}\mathrm{~d} V_{\mathbb{B}^{N}}(x)  \\
&\leq C_{\varepsilon} \int_{d(x,x_{i}) \leq d(x_{i},0)}  \e^{-\delta d(x_{i},0)} \mathrm{~d} V_{\mathbb{B}^{N}}(x)\\
&\leq C_{\varepsilon} \e^{-\delta d(x_{i},0)}\int_{d(x,0) \leq d(x_{i},0)}\mathrm{~d} V_{\mathbb{B}^{N}}(x)\\
&\leq C_{\varepsilon} \e^{-\delta d(x_{i},0)} \e^{(N-1)d(x_{i},0)}.\notag
\end{align*}

\noi Combining the above estimates results in
\begin{align}
&\int_{\mathbb{B}^{N}}(1-a)_{+}\left(u_{1}+u_{2}\right)^{p+1} \mathrm{~d} V_{\mathbb{B}^{N}}(x)\notag\\
&\leq C_{\varepsilon} \sum_{i=1}^{2}\left(\e^{-(c(N,\lambda ) - \varepsilon)(p+1)d(x_{i},0)}+ \e^{-\delta d(x_{i},0)} \e^{(N-1)d(x_{i},0)}\right)\notag\\
&\leq C_{\varepsilon} \sum_{i=1}^{2}\left(\e^{-(\delta-(N-1))d(x_{i},0)}\right) \text{ since } \delta < (p+1)(c(N,\lambda) - \varepsilon) \notag\\
&\leq C_{\varepsilon} \e^{\left(-\frac{K c(N, \lambda)d(x_{1},x_{2})}{K+ R^{\alpha^{\prime}-\alpha}}\right)} \notag \\
&\leq C_{\varepsilon}\e^{\left(-\frac{\left(c(N, \lambda)+c_{K}\right)d(x_{1},x_{2})}{1+\frac{1}{K} R^{\alpha^{\prime}-\alpha}}\right)} \notag \\
  &\leq C_{\varepsilon}\e^{ -\left(c(N, \lambda)+c_{K}\right)d(x_{1}, x_{2})} \notag \\
  & \leq o(1)\left(\int_{\mathbb{B}^{N}} u_{1}^{p} u_{2} \mathrm{~d} V_{\mathbb{B}^{N}}\right) \label{3xx}
&\intertext{for large $R>0$.}  \notag 
\end{align}
 In the above calculations we have used our hypothesis, $d\left(x_{1},x_{2}\right) \leq (K+  R^{\alpha^{\prime}-\alpha}) d(x_{i},0)$ for $i=1,2.$ Further, as $p > 1$, we can use Lemma \ref{lem3b} to get 
\begin{equation}
\int_{\mathbb{B}^{N}}\left(u_{1}+u_{2}\right)^{p+1} \mathrm{~d} V_{\mathbb{B}^{N}} \geq \int_{\mathbb{B}^{N}}\left(u_{1}^{p+1}+u_{2}^{p+1}\right) \mathrm{~d} V_{\mathbb{B}^{N}}+ p \int_{\mathbb{B}^{N}}\left(u_{1}^{p} u_{2}+u_{1} u_{2}^{p}\right) \mathrm{~d} V_{\mathbb{B}^{N}}. \label{3q} 
\end{equation}
Since, $u_{1}, u_{2}$ solves \eqref{3d}, that implies
\begin{equation}
\left\langle u_{1}, u_{2}\right\rangle_{H_{\lambda}}=\int_{\mathbb{B}^{N}} u_{1}^{p} u_{2} \mathrm{~d} V_{\mathbb{B}^{N}}=\int_{\mathbb{B}^{N}} u_{2}^{p} u_{1} \mathrm{~d} V_{\mathbb{B}^{N}}, \;\left\|u_{i}\right\|_{H_{\lambda}}=\int_{\mathbb{B}^{N}} u_{i}^{p+1} \mathrm{~d} V_{\mathbb{B}^{N}}, \; i=1,2. \label{3r}
\end{equation}
Therefore, using \eqref{3r} in \eqref{3q}, we obtain
 \begin{equation}
 \int_{\mathbb{B}^{N}}\left(u_{1}+u_{2}\right)^{p+1} \mathrm{~d} V_{\mathbb{B}^{N}} \geq 2 A+2p\left\langle u_{1}, u_{2}\right\rangle_{H_{\lambda}}. \label{3s}
 \end{equation}
 Combining \eqref{3s} and \eqref{3xx} with \eqref{3.g} yields
 \begin{equation*}
\begin{aligned}
J\left(u_{1}+u_{2}\right) & \leq \frac{2 A+2\left\langle u_{1}, u_{2}\right\rangle_{H_{\la}}}{\left(\int_{\mathbb{B}^{N}}\left(u_{1}+u_{2}\right)^{p+1}\mathrm{~d} V_{\mathbb{B}^{N}}-\int_{\mathbb{B}^{N}}(1-a)_{+}\left(u_{1}+u_{2}\right)^{p+1}\mathrm{~d} V_{\mathbb{B}^{N}}\right)^{\frac{2}{p+1}}} \\
& \leq \frac{2 A+2\left\langle u_{1}, u_{2}\right\rangle_{H_{\la}}}{\left(2 A+\left(2 p-o(1)\right)\left\langle u_{1}, u_{2}\right\rangle_{H_{\la}}\right)^{\frac{2}{p+1}}} \\
&=\frac{(2 A)^{\frac{p-1}{p+1}}\left(1+\frac{1}{A}\left\langle u_{1}, u_{2}\right\rangle_{H_{\la}}\right)}{\left(1+\frac{p-o(1)}{A}\left\langle u_{1}, u_{2}\right\rangle_{H_{\la}}\right)^{\frac{2}{p+1}}}
\end{aligned}
\end{equation*}
From \cite{MS}, it is known that $J_{\infty}$ is achieved by $w$, which is a solution of \eqref{3d}. This in turn implies $S_{1, \lambda}=A^{\frac{p-1}{p+1}}$ and $S_{2, \lambda}=(2 A)^{\frac{p-1}{p+1}}$. Hence,
\begin{equation*}
\begin{aligned}
J\left(u_{1}+u_{2}\right) & \leq S_{2, \lambda} \frac{1+\frac{1}{A}\left\langle u_{1}, u_{2}\right\rangle_{H_{\la}}}{\left(1+\frac{p-o(1)}{A}\left\langle u_{1}, u_{2}\right\rangle_{H_{\la}}\right)^{\frac{2}{p+1}}} \\
& \leq S_{2, \lambda} \frac{1+\frac{1}{A}\left\langle u_{1}, u_{2}\right\rangle_{H_{\la}}}{\left(1+\frac{2 p-o(1)}{(p+1) A}\left\langle u_{1}, u_{2}\right\rangle_{H_{\la}}\right)}.
\end{aligned}
\end{equation*}
Now, $\frac{1}{A}<\frac{2 p-o(1)}{(p+1) A}$ for $R$ large and hence $J\left(u_{1}+u_{2}\right)<S_{2, \lambda}. $
\item We will complete the proof of the lemma in this step. Suppose,
\begin{equation*}
v_{1}=t u_{1}, \quad v_{2}=(1-t) u_{2}, \quad \text { where } \quad t \in[0,1] .
\end{equation*}
Doing a straight forward computation as in \eqref{3.g}, it is easy to check that
\begin{equation*}
J\left(v_{1}+v_{2}\right) \leq \frac{\left(t^{2}+(1-t)^{2}\right) A+2 t(1-t)\left\langle u_{1}, u_{2}\right\rangle_{H_{\la}}}{\left(\int_{\mathbb{B}^{N}}\left|v_{1}+v_{2}\right|^{p+1} \mathrm{~d} V_{\mathbb{B}^{N}}-\int_{\mathbb{B}^{N}}(1-a)_{+}\left|v_{1}+v_{2}\right|^{p+1} \mathrm{~d} V_{\mathbb{B}^{N}}\right)^{\frac{2}{p+1}}}
\end{equation*}
Observing that for $t$ or $1-t$ tending to zero, $v_{1}+v_{2}$ tends to $u_{2}$ or $u_{1}$, and consequently, $J\left(v_{1}+v_{2}\right)$ converges to $S_{1, \lambda}$. Therefore, there exists some $\delta^{\prime}>0$ such that any $\min \{t, 1-t\} \leq \delta^{\prime}$, then $J\left(v_{1}+v_{2}\right)<S_{2, \lambda}$, and this $\delta^{\prime}$ is independent of (large) $R$. Moreover, $J(\frac{u_1 + u_2}{2}) < S_{2, \lambda},$ i.e.,  $J(tu_1 + (1-t) u_2) < S_{2, \lambda}$ for $t = \frac{1}{2}.$ Therefore there exists a neighbourhood, say $N(\frac{1}{2})$ of $t =\frac{1}{2}$ 
such that $J(t u_1 + (1-t)u_2) < S_{2,\la}$ for all $t \in N(\frac{1}{2}).$

From here onward, we will assume $\min \{t, 1-t\} \geq \delta^{\prime}$ and $t \in ([0, 1] \setminus N(\frac{1}{2})).$

\medskip

Using \eqref{3xx}, we can get the subsequent inequality
\begin{equation}
\begin{aligned}
\int_{\mathbb{B}^{N}}(1-a)_{+}\left|v_{1}+v_{2}\right|^{p+1} \mathrm{~d} V_{\mathbb{B}^{N}} &=\int_{\mathbb{B}^{N}}(1-a)_{+}\left|t u_{1}+(1-t) u_{2}\right|^{p+1} \mathrm{~d} V_{\mathbb{B}^{N}} \\
& \leq \int_{\mathbb{B}^{N}}(1-a)_{+}\left|u_{1}+u_{2}\right|^{p+1} \mathrm{~d} V_{\mathbb{B}^{N}}  \\
& \leq o(1)\left\langle u_{1}, u_{2}\right\rangle_{H_{\lambda}}. \label{3l.l}
\end{aligned}
\end{equation} 

\medskip

 Using Lemma~\ref{lem3b} and equations in \eqref{3r}, we obtain 
\begin{equation}
\begin{aligned}
\int_{\mathbb{B}^{N}}\left|v_{1}+v_{2}\right|^{p+1}\mathrm{~d} V_{\mathbb{B}^{N}}=& \int_{\mathbb{B}^{N}}\left|t u_{1}+(1-t) u_{2}\right|^{p+1} \mathrm{~d} V_{\mathbb{B}^{N}} \\
\geq\; & t^{p+1} \int_{\mathbb{B}^{N}} u_{1}^{p+1} \mathrm{~d} V_{\mathbb{B}^{N}}+(1-t)^{p+1} \int_{\mathbb{B}^{N}} u_{2}^{p+1} \mathrm{~d} V_{\mathbb{B}^{N}}\\
&+p \int_{\mathbb{B}^{N}}\left[t^{p}(1-t) u_{1}^{p} u_{2}+t(1-t)^{p} u_{1} u_{2}^{p}\right] \mathrm{~d} V_{\mathbb{B}^{N}} \\
= &\left(t^{p+1}+(1-t)^{p}\right) A+p\left\{t^{p}(1-t)+t(1-t)^{p}\right\}\left\langle u_{1}, u_{2}\right\rangle_{H_{\lambda}}. \\ \label{4j.j}
\end{aligned}
\end{equation}
We claim that 
\begin{equation*}
  \left\langle u_{1}, u_{2}\right\rangle_{H_{\lambda}}=\int_{\mathbb{B}^{N}} u_{1}^{p} u_{2} \mathrm{~d} V_{\mathbb{B}^{N}} \rightarrow 0 \quad \text{as} \quad R \rightarrow \infty. 
\end{equation*}
\begin{equation*}
\begin{aligned}
\int_{\mathbb{B}^{N}} u_{1}^{p} u_{2} \mathrm{~d} V_{\mathbb{B}^{N}} & = \int_{\mathbb{B}^{N}} w\left(d(\tau_{-x_{1}}(x),0)\right)^{p} w\left(d(\tau_{-x_{2}}(x),0)\right)\mathrm{~d} V_{\mathbb{B}^{N}}(x)\\
&= \int_{\mathbb{B}^{N}} w\left(d(x,0)\right)^{p} w\left(d(x,\tau_{-x_{1}}(x_{2}))\right)\mathrm{~d} V_{\mathbb{B}^{N}}(x)\\
&= \int_{d(x,0)\leq \frac{d(\bar{x},0)}{2}} w\left(d(x,0)\right)^{p} w\left(d(x,\tau_{-x_{1}}(x_{2}))\right)\mathrm{~d} V_{\mathbb{B}^{N}}(x)\\
& +  \int_{d(x,0)> \frac{d(\bar{x},0)}{2}}w\left(d(x,0)\right)^{p} w\left(d(x,\tau_{-x_{1}}(x_{2}))\right)\mathrm{~d} V_{\mathbb{B}^{N}}(x)\\
& \text{where}\quad \bar{x} =  \tau_{-x_{1}}(x_{2})\\
&:=I_{1}+I_{2}.\\
\end{aligned}
\end{equation*}
Now, 
\begin{equation*}
\begin{aligned}
I_{1}: &= \int_{d(x,0) \leq \frac{d(\bar{x},0)}{2}}w\left(d(x,0)\right)^{p} w\left(d(x,\tau_{-x_{1}}(x_{2}))\right)\mathrm{~d} V_{\mathbb{B}^{N}}(x). 
\end{aligned}
\end{equation*}
\begin{equation*}
\text{As} \quad d(\bar{x},x) \geq d(\bar{x},0) - d(x,0) \geq  d(\bar{x},0) - \frac{d(\bar{x},0)}{2} = \frac{d(\bar{x},0)}{2}
\end{equation*}
Consequently, $w\left(d(\bar{x},x)\right) \leq w\left(\frac{d(\bar{x},0)}{2}\right)$. Therefore,\\
\begin{equation*}
\begin{aligned}
I_{1}& \leq \int_{d(x,0)\leq \frac{d(\bar{x},0)}{2}}w\left(d(x,0)\right)^{p} w\left(\frac{d(\bar{x},0)}{2}\right)\mathrm{~d} V_{\mathbb{B}^{N}}(x)\\
& \leq C_{\varepsilon} \e^{-(c(N,\lambda) - \varepsilon) \frac{d(\bar{x},0)}{2}} \int_{0}^{\frac{d(\bar{x},0)}{2}} \e^{-(c(N, \lambda) - \varepsilon)p r} \e^{(N-1)r} \mathrm{~d}r\\
& \leq C_{\varepsilon} \left[\e^{\left(-(c(N,\lambda) - \varepsilon)(p+1)+(N-1)\right) \frac{d(\bar{x},0)}{2}}- \e^{-(c(N,\lambda) - \varepsilon)\frac{d(\bar{x},0)}{2}}\right] \, = \, o(1),\\
\end{aligned}
 \end{equation*}
 since $c(N, \lambda) (p+1) > N-1$, we can choose $\varepsilon$ such that $(c(N, \lambda) - \varepsilon) (p+1) > N-1$ and where $o(1)$  tends to  $0$ as $R \rightarrow \infty$. 

\medskip 

Next we shall estimate $I_2 :$  to estimate $I_2$ we  distinguish into  the following two cases.

\medskip

Case 1 : $1 < p < 2.$
 
\begin{equation*}
\begin{aligned}
I_{2}: &= \int_{d(x,0)> \frac{d(\bar{x},0)}{2}}w\left(d(x,0)\right)^{p} w\left(d(x,\tau_{-x_{1}}(x_{2}))\right)\mathrm{~d} V_{\mathbb{B}^{N}}(x)\\
& \leq \left( \int_{d(x,0)> \frac{d(\bar{x},0)}{2}}w^{2}\left(d(x,0)\right)\mathrm{~d} V_{\mathbb{B}^{N}}(x)\right)^{\frac{p}{2}}\left( \int_{d(x,0)> \frac{d(\bar{x},0)}{2}}w^{\frac{2}{2-p}}\left(d(x,\tau_{-x_{1}}(x_{2}))\right)\mathrm{~d} V_{\mathbb{B}^{N}}(x)\right)^{\frac{2-p}{2}}\\
& \leq \left( \int_{d(x,0)> \frac{d(\bar{x},0)}{2}}w^{2}\left(d(x,0)\right)\mathrm{~d} V_{\mathbb{B}^{N}}(x)\right)^{\frac{p}{2}} 
\left( \int_{\mathbb{B}^{N}}w^{\frac{2}{2-p}}\left(d(x,\tau_{-x_{1}}(x_{2}))\right)\mathrm{~d} V_{\mathbb{B}^{N}}(x)\right)^{\frac{2-p}{2}}\\
&\leq C_{\varepsilon} \left( \int_{d(x,0)> \frac{d(\bar{x},0)}{2}} \e^{-(c(N, \lambda) - \varepsilon)2r} \e^{(N-1)r} \, {\rm d}r \right)^{\frac{p}{2}} 
\underbrace{\left(\int_{0}^{\infty} \e^{-(c(N, \lambda) - \varepsilon)\frac{2}{2-p}r} \e^{(N-1)r} \, {\rm d}r \right)^{\frac{2-p}{2}}}_{\leq \, C} \, = \, o(1)
\end{aligned}
\end{equation*}
where $o(1)$  tends to  $0$ as $R \rightarrow \infty.$ i.e., $d(x_1, x_2) \rightarrow \infty,$  in the second last step we have used the fact that
 $\frac{2}{2-p} > 2$ whenever $1 < p < 2,$ and choose $\varepsilon$ such that $(c(N, \lambda) - \varepsilon) \frac{2}{2-p} > N-1.$

\medskip 

Case 2 : $p \geq 2$
\begin{align*}
I_{2}: &= \int_{d(x,0)> \frac{d(\bar{x},0)}{2}}w\left(d(x,0)\right)^{p} w\left(d(x,\tau_{-x_{1}}(x_{2}))\right)\mathrm{~d} V_{\mathbb{B}^{N}}(x) \notag \\
& \leq \left( \int_{d(x,0)> \frac{d(\bar{x},0)}{2}} (w\left(d(x,0))\right)^{\frac{2Np}{N+2}}\mathrm{~d} V_{\mathbb{B}^{N}}(x)\right)^{\frac{N+2}{2N}} \notag \\
& \times \underbrace{ \left( \int_{\bn} \, w^{2^{\star}} 
\left(d(x,\tau_{-x_{1}}(x_{2}))\right)\mathrm{~d} V_{\mathbb{B}^{N}}(x)\right)^{\frac{1}{2^{\star}}}}_{\leq \, C} \notag\\
& \leq C_{\varepsilon} \, \left( \int_{\frac{d(\bar{x}, 0)}{2}}^{\infty} \, e^{- (c(N, \lambda) - \varepsilon) \, \frac{2Np r}{N + 2} + (N-1) r} \, {\rm d}r \right)^{\frac{N+2}{2N}} = o(1), \notag \\
\end{align*}
where in the last step we used the fact that $ (c(N, \lambda) - \varepsilon) \, \frac{2Np}{N + 2} - (N-1) \, > \, 0 $ for $p \geq 2.$  Now by choosing minimum among all finitely many $\varepsilon,$ (in step~1 and step~2) we can make all the above constants uniform. 

\medskip 

Therefore, using \eqref{3l.l} and \eqref{4j.j} we can evaluate the following:
\begin{equation*}
\begin{aligned}
& J\left(v_{1}+v_{2}\right) \\
\leq & \frac{\left(t^{2}+(1-t)^{2}\right) A+2 t(1-t)\left\langle u_{1}, u_{2}\right\rangle_{H_{\lambda}}}{\left[\left(t^{p+1}+(1-t)^{p+1}\right) A+p\left(t^{p}(1-t)+t(1-t)^{p}-o(1)\right)\left\langle u_{1}, u_{2}\right\rangle_{H_{\lambda}}\right]^{\frac{2}{p+1}}} \\
 & \leq \frac{\left(t^{2}+(1-t)^{2}\right) A}{\left(t^{p+1}+(1-t)^{p+1}\right)^{\frac{2}{p+1}} A^{\frac{2}{p+1}}}\\
 &+ \frac{2 t(1-t)\left\langle u_{1}, u_{2}\right\rangle_{H_{\lambda}}}{\left(t^{p+1}+(1-t)^{p+1}\right)^{\frac{2}{p+1}} A^{\frac{2}{p+1}}} \times \left(1+p\left(\frac{\left(t^{p}(1-t)+t(1-t)^{p}\right)}{\left(t^{p+1}+(1-t)^{p+1}\right) A}-o(1)\right)\left\langle u_{1}, u_{2}\right\rangle_{H_{\lambda}}\right)^{\frac{-2}{p+1}}.
\end{aligned}
\end{equation*}
Since $\frac{A}{A^{\frac{2}{p+1}}}=S_{1, \lambda}=2^{\frac{1-p}{1+p}} S_{2, \lambda}$ we have,\\ 
\begin{equation*}
    \begin{aligned}
J\left(v_{1}+v_{2}\right) &\leq   S_{1, \lambda} \times \frac{t^{2}+(1-t)^{2}}{\left(t^{p+1}+(1-t)^{p+1}\right)^{\frac{2}{p+1}}} \\
&+ \frac{2 t(1-t)\left\langle u_{1}, u_{2}\right\rangle_{H_{\lambda}}}{\left(t^{p+1}+(1-t)^{p+1}\right)^{\frac{2}{p+1}} A^{\frac{2}{p+1}}}\\
&\times \left(1+p\left(\frac{\left(t^{p}(1-t)+t(1-t)^{p}\right)}{\left(t^{p+1}+(1-t)^{p+1}\right) A}-o(1)\right)\left\langle u_{1}, u_{2}\right\rangle_{H_{\lambda}}\right)^{\frac{-2}{p+1}} \notag\\
& =  S_{1,\la} \, \frac{t^2 + (1-t)^2}{(t^{p+1} \, + \, (1-t)^{p+1})^{\frac{2}{p+1}}} 
   \,+\, \frac{2t(1-t)\left\langle u_{1}, u_{2}\right\rangle_{H_{\lambda}}}{(t^{p+1} \, + \, (1-t)^{p+1})^{\frac{2}{p+1}} A^{\frac{2}{p+1}}} \times \notag \\
   & \left( 1 - \frac{2p}{p+1} \frac{(t^p(1-t) + t(1-t)^p)}{(t^{p+1} + (1 -t)^{p+1})A} \, \left\langle u_{1}, u_{2}\right\rangle_{H_{\lambda}} \, + \, o(1) \right) \notag \\
  & =  S_{1, \la} \times \frac{t^{2}+(1-t)^{2}}{\left(t^{p+1}+(1-t)^{p+1}\right)^{\frac{2}{p+1}}} + \, C \, g(t)\, \langle u_1, u_2 \rangle_{H_{\lambda}} \, + \, o(\langle u_1 u_2 \rangle_{H_{\lambda}}),
    \end{aligned}
\end{equation*}
where $g(t)$ is a bounded function and $C >0$ is a constant.  Thus we obtain 

\begin{equation*}
J\left(v_{1}+v_{2}\right) \leq S_{1, \lambda} \times \frac{t^{2}+(1-t)^{2}}{\left(t^{p+1}+(1-t)^{p+1}\right)^{\frac{2}{p+1}}} + \, g(t)\, \langle u_1, u_2 \rangle _{H_{\lambda}} \, + \, o(\langle u_1 u_2 \rangle_{H_{\lambda}}).
\end{equation*}

Moreover, we have 

$$
\mbox{max} \, \frac{t^{2}+(1-t)^{2}}{\left(t^{p+1}+(1-t)^{p+1}\right)^{\frac{2}{p+1}}} < 2^{\frac{p-1}{p+1}}, \quad \forall \ t \in ([0, 1]\setminus N(\frac{1}{2})).
$$
Therefore using the above fact and letting $R \rightarrow \infty,$ we conclude 

$$
J(v_1 + v_2) \, < \,  2^{\frac{p-1}{p+1}}S_{1,\la} = S_{2, \lambda}.
$$

\end{enumerate}

This completes the proof of Lemma~\ref{lem3c}.
\end{proof}

\medskip 

\section{Proof of Theorem~\ref{maintheorem1}}\label{pf}

In this section, we shall show the existence of a solution by employing the min-max procedure in the spirit of Bahri-Li \cite{Bahri-Li}. Before going further, let us define some notations 
and functional. 
\medskip 
\begin{equation*}
\Sigma:=\left\{u \in H^{1}\left(\mathbb{B}^{N}\right):\|u\|_{\lambda}=1\right\}, \quad \Sigma^{+}:=\left\{u \in \Sigma: u \geq 0 \quad \text { a.e. in } \quad \mathbb{B}^{N}\right\}. 
\end{equation*}
\medskip 
We shall now define the {\it center of mass \rm}corresponding to $\Sigma$ on the hyperbolic space. To this end, it is worth noticing that 

$$
\dfrac{|x|}{d(x, 0)} := \dfrac{|x|}{\log \left(\frac{1+ |x|}{1- |x|} \right)} \rightarrow 
\begin{cases}
\frac{1}{2}, \quad \quad  x \rightarrow 0 \\
0, \quad \quad  x \rightarrow 1.
\end{cases}
$$\\

Moreover, the function is bounded in $\mathbb{B}^{N},$ i.e., the unit disc in $\mathbb{R}^{N}$. Therefore let $k := \sup \left\{ \dfrac{|x|}{d(x, 0)} : x \in \mathbb{R}^{N} \text{ with } |x| < 1\right \} < \infty.$ Now we are in a 
situation to define the following quantity:\\
Let $m: \Sigma \rightarrow B(0,1)$ defined as

\begin{equation*}
    m(u):=\frac{\tanh (1/2)}{\|u\|_{L^{p+1}\left(\mathbb{B}^{N}\right)}^{p+1}} \int_{\bn} \frac{x}{k\, d(x,0)}|u|^{p+1} \mathrm{~d} V_{\mathbb{B}^{N}}.
    \end{equation*}
    Clearly $|m(u)| < 1.$ As discussed in Section~2, and using elementary computation we can also conclude that $d(m(u),0) \leq 1.$ Moreover, $m$ is continuous from $\Sigma$ to $\mathbb{B}^{N}.$

\medskip 

\textbf{We can now prove Theorem \ref{maintheorem1} as follows:}

\begin{proof}

Let  $J$ as defined in \eqref{3b}, we define

\begin{equation*}
I_{z}:=\inf _{m(u)=z,\; u \in \Sigma} J(u) \quad \text{for} \quad z \in \mathbb{B}^{N} \quad \text{such that} \quad  d(z,0)<1. 
\end{equation*}
\medskip

 It is straightforward to note that if inf $_{u \in \Sigma} J(u)<S_{1, \lambda},$ then exploiting  the standard variational arguments (see \cite{MSV}), 
 the existence of a positive solution of $(P_\la)$ can be proven, which in fact, is  a constant multiple of the global minimum of the functional $J$.
 
 \medskip

Henceforth we are left with the case when $\inf _{u \in \Sigma} J(u) \geq S_{1, \lambda}$.

\medskip

If $I_{z}=S_{1, \lambda}$ for some $z \in \mathbb{B}^{N}$ with $d(z,0)<1$ then there exists some $u \in \Sigma^{+}, m(u)=z$, such that $J(u)=S_{1, \lambda}$. The theorem, under this condition, follows from the concentration compactness principle established in \cite{MSV} in the spirit of \cite{Bahri-Li}, with minor modifications needed for the hyperbolic space
and again, in this case, the solution obtained is a constant multiple of the global minimum of $J.$

\medskip

Thus we are only left with the following possibility
$$I_{z}>S_{1, \lambda}\text{ for every}\quad z \in \mathbb{B}^{N}\text{ with }d(z,0)<1.$$
We now claim the existence of some positive constant $R_{1}$ for which the following holds
$$J(w(\tau_{-y}(\bullet))) \leq S_{1, \lambda}\quad \text{for every} \quad y \in \mathbb{B}^{N} \quad \text{with} \quad d(y,0) \geq R_{1}.$$
To prove the above claim, we can observe the following by utilizing (a2):
\begin{equation*}
 \liminf\limits_{d(y,0)\rightarrow \infty}a(\tau_{y}(x)) \geq 1 \quad \mbox{for every} \ x \in \mathbb{B}^{N}.
\end{equation*}
Further, applying Fatou's lemma and the above equation yields

\begin{equation}
\begin{aligned}
    \int_{\mathbb{B}^{N}}|w(x)|^{p+1} \mathrm{~d} V_{\mathbb{B}^{N}} &\leq \int_{\mathbb{B}^{N}} \liminf\limits_{d(y,0)\rightarrow \infty}a(\tau_{y}(x))|w(x)|^{p+1} \mathrm{~d} V_{\mathbb{B}^{N}}\\
    & \leq \liminf\limits_{d(y,0)\rightarrow \infty} \int_{\mathbb{B}^{N}}  a(\tau_{y}(x))|w(x)|^{p+1} \mathrm{~d} V_{\mathbb{B}^{N}}. \label{4.ut}
    \end{aligned} 
\end{equation}

\medskip 

Thus we have, 
\begin{align*}
 \lim _{d(y,0) \rightarrow \infty} J(w(\tau_{-y}(\bullet))) & =   \frac{\|w\|_{\lambda}^{2}}{\liminf\limits_{d(y,0)\rightarrow \infty}\left(\int_{\mathbb{B}^{N}} a(\tau_{y}(x))|w(x)|^{p+1} \mathrm{~d} V_{\mathbb{B}^{N}}\right)^{\frac{2}{p+1}}}\\
 & \leq \frac{\|w\|_{\lambda}^{2}}{\left(\int_{\mathbb{B}^{N}}|w(x)|^{p+1} \mathrm{~d} V_{\mathbb{B}^{N}}\right)^{\frac{2}{p+1}}}\\
 & = J_{\infty}(w)= S_{1, \lambda}, \\
\end{align*}
  where we have used \eqref{4.ut} to get the last inequality. Hence the claim follows.
 \medskip

Fixing any $z \in \mathbb{B}^{N}$ such that $d(z,0)<1$ and $I_{z}>S_{1, \lambda}$. Consequently, $S_{1, \lambda}<\frac{1}{2}\left(I_{z}+S_{1, \lambda}\right)<I_{z}$. 
Thus, using above the claim, there exists some positive constant $\bar{R}_{1}$, such that
\begin{equation*}
I_{z}>\frac{1}{2}\left(I_{z}+S_{1, \lambda}\right)>S_{1, \lambda} \geq J(w(\tau_{-y}(\bullet))) \quad \forall y \quad \text{with} \quad d(y,0) \geq \bar{R}_{1}. 
\end{equation*}

Assume $\alpha, \alpha^{\prime}$ and $R>R_{0}$ be the same constants as were in lemma \ref{lem3c}. Choose $R_{2}>\max \left\{R^{\alpha}, \bar{R}_{1}\right\}$be very large and
\begin{equation*}
x_{2}=\left(0,0, \ldots, \tanh{\left(\frac{R_{2}-R_{2}^{\frac{\alpha^{\prime}}{\alpha}}}{2}\right)}\right).
\end{equation*} 
Define a map $h_{0}: \partial B(0,R_{2}) \rightarrow \Sigma^{+}$ as 
\begin{equation*}
h_{0}\left(x_{1}\right)=\frac{w\left(\tau_{-x_{1}}(\bullet)\right)}{\left\|w\left(\tau_{-x_{1}}(\bullet)\right)\right\|_{\la}}, \text { where } x_{1} \in \partial B(0,R_{2}).
\end{equation*}
By the choice of $R_{2}$, we have $R_{2}>\bar{R}_{1}$, so we get
\begin{equation*}
J\left(h_{0}\left(x_{1}\right)\right)<\frac{1}{2}\left(I_{z}+S_{1, \lambda}\right)<I_{z} \quad \text{for} \quad x_{1} \in \partial B(0,R_{2}). 
\end{equation*} 
Now we define another map $h^{*}: B(0,R_{2}) \rightarrow \Sigma^{+}$ by
\begin{equation*}
h^{*}\left(t x_{1}+(1-t) x_{2}\right)=\frac{t w\left(\tau_{-x_{1}}(\bullet)\right)+(1-t) w\left(\tau_{-x_{2}}(\bullet)\right)}{\left\|t w\left( \tau_{-x_{1}}(\bullet)\right)+(1-t) w\left(\tau_{-x_{2}}(\bullet)\right)\right\|_{\la}},
\end{equation*}
for $0 \leq t \leq 1,\;x_{1} \in \partial B(0,R_{2})$. It can be observed that $\left.h^{*}\right|_{\partial B(0,R_{2})}=h_{0}$.\\
Also, from Lemma \ref{lem3c}, we obtain
\begin{equation}
J\left(h^{*}(y)\right)<S_{2, \lambda} \text { for all } y \in B(0,R_{2}). \label{4e}
\end{equation}
Next, we define some min-max value. Let
\begin{equation*}
\Gamma:=\left\{h: B(0,R_{2}) \rightarrow \Sigma^{+}: h \text { is continuous, }\left.h\right|_{\partial B(0,R_{2})}=h_{0}\right\} 
\end{equation*}
and
\begin{equation*}
\mu_{0}:=\inf _{h \in \Gamma} \max _{y \in B(0,R_{2})} J(h(y)). 
\end{equation*}
We now claim that the following holds 
\begin{equation}
S_{1, \lambda}<I_{z} \leq \mu_{0}<S_{2, \lambda}. \label{4.s}
\end{equation}
Using \eqref{4e}, it is evident that $ \mu _ { 0 } < S _ { 2, \lambda } $. Furthermore, consider the map 
\begin{equation*}
m \circ h: B(0,R_{2}) \rightarrow \mathbb{B}^{N}, \text { where } h \in \Gamma .
\end{equation*}

Then one can see that 
\begin{equation*}
\lim _{R_{2} \rightarrow \infty} m \circ h_{0}\left(x_{1}\right)= \frac{x_{1}}{k d(x_{1},0)} \tanh{\left(1/2\right)} \text { uniformly for } x_{1} \in \partial B(0,R_{2}).
\end{equation*}
Thus, the above convergence immediately implies that for any $z \in \mathbb{B}^{N}$ with $d(z,0) < d(m \circ h_{0}\left(x_{1}\right),0 )$ for $x_{1} \in \partial B(0,R_{2}),$ 
 $z \notin m \circ h\left(\partial B(0,R_{2})\right)$ and hence $\operatorname{deg}\left(m \circ h,\; B(0,R_{2}),\; z\right)$
  is well defined.
\medskip

Let us define $I: \bar B(0,R_{2}) \rightarrow B(0,1)$ by $$I(x):=\frac{x}{k R_{2}} \tanh(1/2).$$ 
Moreover, for $R_{2}$ large enough, $m \circ h\left(\partial B(0,R_{2})\right)=I\left(\partial B(0,R_{2})\right)$. 
Therefore applying degree theory, for $R_{2}$ large enough, we have
\begin{equation*}
\operatorname{deg}\left(m \circ h,\; B(0,R_{2}),\; z\right)=\operatorname{deg}\left(I,\; B(0,R_{2}),\; z\right)=1.
\end{equation*}
In particular, using the solution property of degree, we can guarantee the existence of some  $y \in B(0,R_{2})$, such that, $m \circ h(y)=z$. 
Hence $\mu_{0} \geq I_{z}>S_{1, \la}$.

\medskip

From Proposition \ref{prop1}, we can conclude that $I_{\lambda,a}$ satisfies PS condition for $\frac{p-1}{2(p+1)} A < c<\frac{p-1}{p+1} A .$ Equivalently, $\left.J\right|_{\Sigma+}$ satisfies PS condition for $S_{1, \la}<c<S_{2, \la}$, i.e., $A^{\frac{p-1}{p+1}}<c<(2 A)^{\frac{p-1}{p+1}}$. Therefore, by \eqref{4.s}, $\left.J\right|_{\Sigma^{+}}$satisfies PS condition at $\mu_{0}$. As a result, using deformation lemma, $\mu_{0}$ is a critical value of $\left.J\right|_{\Sigma+}$ with some corresponding critical point $0 \not \equiv u \geq 0$ i.e., $\left.J^{\prime}\right|_{\Sigma+} (u)=0 .$ Hence we obtain a non-negative solution of $(P_\la)$ by scaling $u$. Furthermore, $u$ is a positive solution to $(P_\la)$ follows from the maximum principle. 
\end{proof}

\par\bigskip\noindent
\textbf{Acknowledgments.}
D.~Ganguly is partially supported by the INSPIRE faculty fellowship (IFA17-MA98).  
D.~Gupta is supported by the PMRF. D. Ganguly is grateful to
D. Karmakar for useful discussions.

\end{document}